\documentclass[11pt]{article}

\usepackage{fullpage}
\usepackage{amsmath}
\usepackage{amstext}
\usepackage{amssymb}
\usepackage{amsthm}
\usepackage{amsfonts}
\usepackage{amsxtra}
\usepackage{mathrsfs}
\usepackage{mathtools}
\usepackage{dsfont}
\usepackage{breqn}
\usepackage{enumerate}
\usepackage{bbm}
\usepackage{comment}
\usepackage[utf8]{inputenc}

\usepackage{sty}

\newtheorem{theorem}{Theorem}[section]
\newtheorem{lemma}[theorem]{Lemma}

\newtheorem{problem}[theorem]{Problem}

\newtheorem{proposition}[theorem]{Proposition}
\newtheorem{corollary}[theorem]{Corollary}

\newtheorem{definition}[theorem]{Definition}

\usepackage{hyperref}
\usepackage{xcolor}
\definecolor{OliveGreen}{HTML}{3C8031}
\hypersetup{
colorlinks=true,
urlcolor=blue,
linkcolor=blue,
citecolor=OliveGreen,
}

\title{Detection-Recovery Gap for Planted Dense Cycles}

\usepackage{authblk}
\author[1]{Cheng Mao\thanks{Email: \textit{cheng.mao@math.gatech.edu}.}}
\author[2]{Alexander S.\ Wein\thanks{Email: \textit{aswein@ucdavis.edu}.}}
\author[1]{Shenduo Zhang\thanks{Email: \textit{szhang705@gatech.edu}.}}

\affil[1]{School of Mathematics, Georgia Institute of Technology}

\affil[2]{Department of Mathematics, University of California, Davis}

\date{}

\begin{document}

\maketitle

\begin{abstract}
Planted dense cycles are a type of latent structure that appears in many applications, such as small-world networks in social sciences and sequence assembly in computational biology. We consider a model where a dense cycle with expected bandwidth $n \tau$ and edge density $p$ is planted in an Erd\H{o}s--R\'enyi graph $G(n,q)$. We characterize the computational thresholds for the associated detection and recovery problems for the class of low-degree polynomial algorithms. In particular, a gap exists between the two thresholds in a certain regime of parameters. For example, if $n^{-3/4} \ll \tau \ll n^{-1/2}$ and $p = C q = \Theta(1)$ for a constant $C>1$, the detection problem is computationally easy while the recovery problem is hard for low-degree algorithms.
\end{abstract}

\tableofcontents

\section{Introduction}
\label{sec:intro}

Recovering latent structures in networks is a broad class of problems that are essential both in theory and for applications in the social and biological sciences \cite{watts2004new,barabasi2012science}.
In this work, we study the detection and recovery of a hidden cyclic structure in an observed network, a type of structure found in many real-world applications. 
For example, the celebrated \emph{Watts--Strogatz small-world model} \cite{watts1998collective} assumes that $n$ nodes have latent positions on a circle, and they have stronger connections with their $k$-nearest neighbors and weaker connections with all other nodes. 
Observing such a small-world network, the problem of interest is to recover the relative positions of the nodes---which nodes are $k$-nearest neighbors of each other---and hence the overall structure of the network.
Since its proposal, the Watts--Strogatz model has been used extensively to study, for example, epidemic behavior \cite{moore2000epidemics}, collaboration networks \cite{uzzi2005collaboration}, and brain networks \cite{bassett2006small}. 
More generally, the problem of recovering a one-dimensional embedding of $n$ objects from pairwise similarities between them arises in a wider range of applications, including relative dating in archaeology \cite{robinson1951method}, \emph{de novo} genome assembly in computational biology \cite{lieberman2009comprehensive}, and angular synchronization in tomography \cite{singer2011angular}. 

Despite the vast literature on related models and algorithms, the statistical and computational limits of this problem are not yet well-established in a rigorous framework. 
The information-theoretic thresholds for the Watts--Strogatz model are studied in \cite{cai2017detection}, but the upper bounds achieved by computationally efficient algorithms are far from the information-theoretic thresholds. 
It is unknown whether these statistical-to-computational gaps are inherent, or whether they can be closed by other efficient algorithms.
In the case where the bandwidth $k$ is at most $n^{o(1)}$, sharp characterizations of recovery conditions are given in \cite{bagaria2020hidden,ding2020consistent} under a more general model. 
Moreover, several other algorithms and analyses have been introduced for related models from the perspective of graphon estimation by
\cite{janssen2022reconstruction,natik2021consistency,giraud2021localization}. 
However, none of the previous works have shown computational lower bounds against a class of efficient algorithms. 

Moreover, the Watts--Strogatz small-world model can be seen as modeling a one-dimensional, noisy \emph{random geometric graph} with latent locations on a circle.
Random geometric graphs have long been studied in a variety of scientific fields; see, e.g., \cite{penrose2003random} and a recent survey by \cite{duchemin2022random}.
In particular, detection or testing thresholds for high-dimensional, noiseless random geometric graphs were studied by \cite{bubeck2016testing} and improved by \cite{brennan2020phase,liu2022testing}.
Variants of the model with edge noise have also been studied recently by \cite{liu2021phase,liu2021probabilistic}. Recovery or reconstruction of the latent geometry from a random geometric graph has also been long studied in various models, especially using spectral techniques; see, e.g., \cite{sussman2013consistent,araya2019latent,eldan2022community}.
Despite the vast literature, the discrepancy between detection and recovery thresholds is yet to be understood in a single model.

In this work, we propose a variant of the Watts--Strogatz small-world model, which is a random graph with a \emph{planted dense cycle}, and study the computational complexities of the associated detection and recovery problems in the framework of \emph{low-degree polynomial algorithms}. 
This framework has proven to be successful at probing the computational complexity of detecting and estimating hidden structures in high-dimensional settings \cite{hopkins2017efficient,kunisky2019notes,schramm2020computational} and is closely related to the sum-of-squares hierarchy \cite{sos-hidden,sam-thesis}. 
For problems such as planted clique, community detection, and sparse PCA, the conjectured hard regime where no polynomial-time algorithms are known to exist coincides with the regime where low-degree polynomials fail to solve the problem. 
For the planted dense cycle problem, we identify the regimes where low-degree polynomial algorithms fail to detect and recover the hidden cycle respectively. 
In particular, we show that the threshold for detection is drastically different from that for recovery, so there is a \emph{detection-recovery gap} for this problem.

\paragraph{Notation}
Let $[n] := \{1,2,\dots,n\}$ and $\binom{[n]}{2} := \{(i,j) : i,j \in [n], \, i<j\}$.
We use the standard asymptotic notation $O(\cdot), o(\cdot), \Omega(\cdot), \omega(\cdot)$, $\Theta(\cdot)$ as $n \to \infty$, and a tilde is added if the asymptotic relation holds up to a polylogarithmic factor in $n$.

Any subset $\alpha \subseteq \binom{[n]}{2}$ can be identified with the graph on vertex set $[n]$ induced by edges in $\alpha$.
Therefore, we can say ``graph $\alpha$'' without ambiguity.
Then $|\alpha|$ denotes the number of edges in the graph $\alpha$.
Let $V(\alpha) \subseteq [n]$ denote the vertex set of $\alpha$, i.e., the set of vertices $v \in [n]$ that are non-isolated by the edges of $\alpha$.

\section{Models and main results}
\label{sec:result}

\subsection{Planted dense cycles}
We now formally introduce our models. 
For any $a,b \in [0,1]$, define 
$$
\dist(a,b) := \min\{|a-b|, 1-|a-b|\} .
$$
In other words, $\dist(a,b)$ is the distance between $a$ and $b$ on a circle of circumference $1$.
Throughout the paper, we consider the setting where the number of vertices $n$ grows, and other parameters $p$, $q$, and $\tau$ may depend on $n$.

\begin{definition}[Model $\cP$, Planted Dense Cycle] \label{def:mod-p}
Suppose that $0 \le q < p \le 1$ and $0 \le \tau \le 1/2$. 
Let $z \in [0,1]^n$ be a latent random vector whose entries $z_1, \dots, z_n$ are i.i.d.\ $\Unif([0,1])$ variables.  
We observe an undirected graph with adjacency matrix $A \in \R^{n \times n}$ whose edges, conditional on $z_1, \dots, z_n$, are independently sampled as follows: 
$A_{ij} \sim \Bern(p)$ if $\dist(z_i,z_j) \le \tau/2$ and $A_{ij} \sim \Bern(q)$ otherwise, where $(i,j) \in \binom{[n]}{2}$.
We write $A \sim \cP$.
\end{definition}

\noindent
In short, a graph $A$ from model~$\cP$ is a $G(n,q)$ \ER graph with a planted dense cycle that has edge density $p$ and expected bandwidth $n \tau$.
The location of the cycle is determined by the latent variable $z$.
For comparison, the Watts--Strogatz model plants a dense cycle of bandwidth \emph{exactly} $n \tau$; it also assumes that the average degree is matched to that in the noiseless case where $p = 1$ and $q = 0$, so $\tau = \tau p + (1-\tau) q$ in \cite{watts1998collective,cai2017detection}.
Moreover, the bandwidth $n \tau$ is typically much smaller than $n$ in small-world networks, so we may assume $\tau \le 1/2$ throughout the paper to ease the presentation.

In addition, we use $\cQ$ to denote an \ER graph model.

\begin{definition}[Model $\cQ$, \ER graph] \label{def:mod-q}
Suppose that $0 \le q < p \le 1$ and $0 \le \tau \le 1/2$. 
Let $r := \tau p + (1-\tau) q$.
We observe a $G(n,r)$ \ER graph with adjacency matrix $A \in \R^{n \times n}$.
We write $A \sim \cQ$.
\end{definition}

\noindent
Note that the condition $r = \tau p + (1-\tau) q$ is imposed so that the average degrees are matched in the two models $\cP$ and $\cQ$.

There are two problems associated with the model of planted dense cycle, detection and recovery.
Detection of the planted cycle is formulated as a statistical hypothesis testing problem.

\begin{problem}[Detection]
\label{prob:detect}
Observing the adjacency matrix $A \in \R^{n \times n}$ of a graph, we test $A \sim \cP$ against $A \sim \cQ$.
\end{problem}

Recovery of the planted cycle is formulated as determining whether vertices $i$ and $j$ are neighbors in the cycle for $(i,j) \in \binom{[n]}{2}$, i.e., whether $\dist(z_i,z_j) \le \tau/2$.
By symmetry, it suffices to consider the pair of vertices $(1,2)$ and estimate $\bbone\{ \dist(z_{1}, z_{2}) \le \tau/2 \}$. 

\begin{problem}[Recovery]
\label{prob:recover}
Observing the adjacency matrix $A \in \R^{n \times n}$ of a graph $A \sim \cP$ with a planted cycle, we aim to recover $\chi := \bbone\{\dist(z_1, z_2) \le \tau/2\}$. 
\end{problem}

\subsection{Overview of results}
Our results fall within the framework of low-degree polynomial algorithms (see \cite{kunisky2019notes}).
Let $\R[A]_{\le D}$ denote the set of multivariate polynomials in the entries of $A$ with degree at most $D$.
The scaling of $D = D_n$ will be made precise later, but in general, when we speak of a ``low-degree'' polynomial, its degree is at most $D = n^{o(1)}$. 

For the detection problem, we study the ability of such a polynomial to distinguish the two distributions $\cP$ and $\cQ$, in the following sense \cite[Definition~1.6]{fp}.

\begin{definition}[Strong separation]
\label{def:strong-separation}
A polynomial $f = f_n \in \RR[A]_{\le D}$ is said to \emph{strongly separate} $\cP$ and $\cQ$ over $A$ if
\[ \sqrt{\Var_\cP[f(A)] \vee \Var_\cQ[f(A)]} = o\left(\left|\E_\cP[f(A)] - \E_\cQ[f(A)]\right|\right) \]
as $n \to \infty$.
\end{definition}
\noindent
By Chebyshev's inequality, strong separation implies that, by thresholding the value $f(A)$, one can test between $A \sim \cP$ and $A \sim \cQ$ with both type I and type II errors of order $o(1)$.

\medskip

For the recovery problem, recall that we aim to estimate $\chi = \bbone\{\dist(z_1, z_2) \le \tau/2\}$.
The quantity of interest is the degree-$D$ minimum mean squared error (see \cite{schramm2020computational})
$$
\MMSE_{\le D} := \inf_{f \in \R[A]_{\le D}} \E_{\cP} \left[ (f(A) - \chi)^2 \right] . 
$$
It is equivalent to consider the degree-$D$ maximum correlation
\begin{equation}
\Corr_{\le D} 
:= \sup_{\substack{f \in \R[A]_{\le D}, \\ \E_{\cP}[f(A)^2] \ne 0}} \frac{\E_{\cP}[f(A) \cdot \chi]}{\sqrt{\E_{\cP}[f(A)^2]}} 
\label{eq:corr-d}
\end{equation}
because of the following relation \cite[Fact~1.1]{schramm2020computational}
$$
\MMSE_{\le D} = \E_{\cP}[\chi^2] - \Corr_{\le D}^2 .
$$
The trivial estimator $f(A) \equiv \E_{\cP}[\chi]$ of $\chi$ achieves a correlation
$$
\frac{\E_{\cP}[f(A) \cdot \chi]}{\sqrt{\E_{\cP}[f(A)^2]}} = \E_{\cP}[\chi] ,
$$
which motivates the following definition.

\begin{definition}[Weak recovery]
\label{def:weak-recovery}
A polynomial $f = f_n \in \RR[A]_{\le D}$ is said to \emph{weakly recover} an estimand $\chi$ given $A \sim \cP$ if
\[ \frac{\E_{\cP}[f(A) \cdot \chi]}{\sqrt{\E_{\cP}[f(A)^2]}} = \omega\left( \E_{\cP}[\chi] \right) \]
as $n \to \infty$.
\end{definition}
\noindent
Note that for the estimand $\chi = \bbone\{\dist(z_1, z_2) \le \tau/2\}$ in Problem~\ref{prob:recover}, we have $\E_{\cP}[\chi] = \tau$.

\medskip

For both the detection and the recovery problem, we establish low-degree upper and lower bounds that match up to an $n^\delta$ factor for an arbitrarily small constant $\delta > 0$. Our main results are summarized in the following theorem.

\begin{theorem}[Summary of the detection-recovery gap] \label{thm:main-informal}
Suppose that $C q \le p \le C' q$ for constants $C' > C > 1$.
Fix any constant $\delta \in (0,0.1)$.
Suppose that $2/\delta \le D \le o \Big( \big(\frac{\log n}{\log \log n}\big)^2 \Big)$ and $\tau \le (\log n)^{-8}$.
\begin{itemize}
\item
\textbf{(Detection)} Consider Problem~\ref{prob:detect} and Definition~\ref{def:strong-separation}.
If $n^3 p^3 \tau^4 \le n^{-\delta}$, then no polynomial in $\R[A]_{\le D}$ strongly separates $\cP$ and $\cQ$.
If $n^3 p^3 \tau^4 = \omega(1)$, then there is a polynomial in $\R[A]_{\le D}$ that strongly separates $\cP$ and $\cQ$.

\item
\textbf{(Recovery)} Consider Problem~\ref{prob:recover} and Definition~\ref{def:weak-recovery}. 
If $n p \tau^2 \le n^{-\delta}$, then no polynomial in $\R[A]_{\le D}$ weakly recovers $\chi$.
If $n p \tau^2 \ge n^{\delta}$, then there is a polynomial in $\R[A]_{\le D}$ that weakly recovers $\chi$.
\end{itemize}
\end{theorem}

\begin{proof}
The four bounds are established in Theorems~\ref{thm:detect-lower}, \ref{thm:detect-upper}, \ref{thm:recover-lower}, and~\ref{thm:recover-upper}, respectively.
It suffices to note that under the assumptions of the theorem, the conditions \eqref{eq:detect-lower-cond}, \eqref{eq:detect-upper-cond}, \eqref{eq:recover-lower-cond}, and \eqref{eq:recover-upper-cond} are all satisfied.
See also the discussion after each of the theorems.
\end{proof}

By the above theorem, there is a gap between the detection threshold and the recovery threshold for planted dense cycles if we focus on low-degree polynomials.
To better illustrate the detection-recovery gap, let us suppose $p = n^{-a}$ and $\tau = n^{-b}$ for constants $a, b \in (0,1)$.
Then the detection threshold is given by $3 - 3a - 4b = 0$, while the recovery threshold is given by $1- a - 2b = 0$.
We plot the phase diagram in Figure~\ref{fig:phase}.
In particular, in region $B$ of the figure, detection is easy while recovery is hard.

% PHASE DIAGRAM

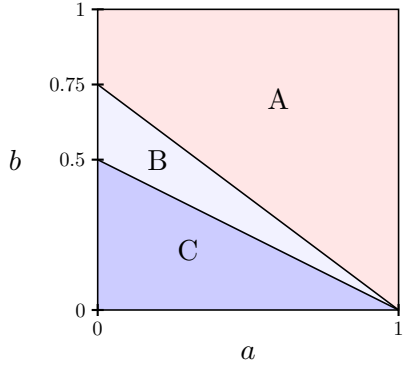
\begin{figure}[ht]
\centering

\begin{tikzpicture}[scale=0.4]

  \def\xtick#1#2{\draw[thick] (#1)++(0,.2) --++ (0,-.4) node[below=-.5pt,scale=0.7] {#2};}
  \def\ytick#1#2{\draw[thick] (#1)++(.2,0) --++ (-.4,0) node[left=-.5pt,scale=0.7] {#2};}
  
  % COORDINATES
  \coordinate (O) at (0,0);
  \coordinate (NE) at (10,10);
  \coordinate (NW) at (0,10);
  \coordinate (SE) at (10,0);
  \coordinate (W1) at (0,5);
  \coordinate (W2) at (0,7.5);
  \coordinate (S1) at (5,0);
  
  % PATHS
  \def\DE{(W2) to (SE)}
  \def\RE{(W1) to (SE)}
  \path[name path=DE] \DE;
  \path[name path=RE] \RE;
  
  % REGIONS
  \fill[mylightred] (W2) -- (SE) -- (NE) -- (NW) -- cycle;
  \fill[blue!5] (W1) -- (SE) -- (W2) -- cycle;
  \fill[mylightblue] (O) -- (SE) -- (W1) -- cycle;

  % REGION NAMES
  \node at (6,7) {A};
  \node at (2,5) {B};
  \node at (3,2) {C};
  
  % LINES
  \draw[semithick] \DE;
  \draw[semithick] \RE;
  
  % AXES
  \draw[semithick] (O) rectangle (NE);
  \node[left=25pt] at (W1) {$b$};
  \node[below=10pt] at (S1) {$a$};
  \xtick{O}{0}
  \xtick{SE}{1}
  \ytick{O}{0}
  \ytick{NW}{1}
  \ytick{W1}{0.5}
  \ytick{W2}{0.75}
  
\end{tikzpicture}

\caption{The detection-recovery gap for planted dense cycles with $p = n^{-a}$ and $\tau = n^{-b}$. Detection is hard in region A, and easy in regions B and C. Recovery is hard in regions A and B, and easy in region C.}
\label{fig:phase}
\end{figure}

In Theorem~\ref{thm:main-informal}, we have assumed that $p$ and $q$ are of the same order, which is a standard simplification in the literature for related problems (see, e.g., \cite{hajek2015computational}).
In fact, for three of the four bounds (Theorems~\ref{thm:detect-lower}, \ref{thm:detect-upper}, and~\ref{thm:recover-lower}, except the recovery upper bound), the edge density $p$ in the cycle can be much higher than the edge density $q$ outside the cycle; for the detection and recovery lower bounds (Theorems~\ref{thm:detect-lower} and~\ref{thm:recover-lower}), $p$ can also approach $q$ in the sense that $p-q$ is of a smaller order than $p$ or $q$.
The latter regime is also addressed in a related but different context of computational lower bounds by \cite{brennan2019reducibility}.

For the recovery upper bound, our proposed statistic and analysis yield stronger results than weak recovery if we consider efficient algorithms beyond low-degree polynomials.
Namely, we produce an estimator $\hat \chi \in \{0,1\}$ that recovers $\chi = \1\{\dist(z_1,z_2) \le \tau/2\}$ with high probability, and also a consistent estimator of the underlying random geometric graph.
See Theorem~\ref{thm:exact-recovery} and Corollary~\ref{cor:rgg}.

\paragraph{Technical contributions}
It is worth noting that none of the four bounds follow trivially from existing work.
For the detection lower bound, while the framework of low-degree polynomials is well-understood (see, e.g., \cite{sam-thesis,fp}), we provide a new application to random geometric graphs.
The analysis also prompts us to study the signed triangle count proposed by \cite{bubeck2016testing} in the noisy case, proving the detection upper bound.
For the recovery lower bound, we generalize the technique developed by \cite{schramm2020computational} for the planted clique problem to a general binary observation model, and then apply it to our problem.
Finally, the most technical part of this work is the recovery upper bound, where we provide a delicate analysis of self-avoiding walks between two vertices in the observed graph. 
The same statistic has been used by \cite{hopkins2017efficient} for community detection, but we perform a new analysis of certain probabilistic and combinatorial properties of self-avoiding walks in random geometric graphs.
% See Sections~\ref{sec:recover} and~\ref{sec:recover-upper-proof}.

\paragraph{Open problems}
While we have characterized the detection and recovery thresholds for low-degree polynomial algorithms, the information-theoretic thresholds for both problems remain largely open.
Most existing results in the literature of small-world graphs or random geometric graphs focus on different regimes and are not comparable to our results.
For example, \cite{ding2020consistent} consider the regime $\tau n = n^{o(1)}$ and \cite{liu2021phase} assume a constant $p$.
One possible exception is the work by \cite{cai2017detection}, which assumes a bandwidth exactly $\tau n$ instead of $\tau n$ in expectation.
Ignoring this difference, their results can be compared to ours for $p$, $q$, and $\tau$ all of the order $n^{-a}$, i.e., on the diagonal $a=b$ in Figure~\ref{fig:phase}.
One of their results states that detection is information-theoretically possible if $a < 1/2$.
Consequently, the information-theoretic threshold for detection would be inside region A in Figure~\ref{fig:phase}, and there would be a statistical-to-computational gap for the detection problem.
However, since the comparison between \cite{cai2017detection} and our work is not fully rigorous, we leave the study of information-theoretic thresholds to future work.

Another interesting question left open by our work is what the detection and recovery thresholds are for higher-dimensional geometry.
For example, the latent locations $z_1, \dots, z_n$ may be distributed on the unit sphere $\mathcal{S}^{d-1}$ in $\R^d$ for $d \ge 3$, rather than on a circle.
We believe many of the results in this work extend to the case of a fixed $d$, but if $d$ grows with $n$, then the problem becomes significantly more difficult and novel ideas are required.

\medskip

In the sequel, we present low-degree lower bounds before upper bounds for both the detection and the recovery problem.
The rationale behind this nonstandard order of presentation is in fact an important advantage of the low-degree framework: The proof of a low-degree lower bound will naturally suggest an efficient algorithm that potentially achieves the matching upper bound.

\section{The detection problem}
\label{sec:detect}

As discussed in the introduction, the planted dense cycle model is a one-dimensional random geometric graph model.
Detection of geometry in random graphs has been studied, and a canonical algorithm for this task is counting \emph{signed triangles} proposed by \cite{bubeck2016testing}.
On the other hand, computational lower bounds for random geometric graphs are not well-understood even in the one-dimensional case.
We first present the low-degree lower bounds, whose proof suggests that the statistic of signed triangles has the best distinguishing power.
Then we give a self-contained analysis of signed triangles in our case.

\subsection{Lower bound}
The standard procedure for proving low-degree lower bounds consists in analyzing the distinguishing power of an orthonormal basis of functions of the observations under model~$\cQ$.
Towards this end, for $(i,j) \in \binom{[n]}{2}$, define
\begin{equation}
\bar A_{ij} := \frac{A_{ij} - r}{\sqrt{r(1-r)}} .
\label{eq:def-a-bar}
\end{equation}
For $\alpha \subseteq \binom{[n]}{2}$, define
\begin{equation*}
\phi_\alpha(A) := \prod_{(i,j) \in \alpha} \bar{A}_{ij} ,
\end{equation*}
and let $\phi_\varnothing(A) \equiv 1$. Then $\{\phi_\alpha\}_{\alpha \subseteq \binom{[n]}{2}}$ is an orthonormal basis for functions on the hypercube $\{0,1\}^{\binom{[n]}{2}}$ under $\cQ$.
Moreover, since $r$ is the average edge density in both models $\cP$ and $\cQ$, the larger $p$ is compared to $r$, the larger signal we have at each edge.
Hence we define a quantity
\begin{equation}
\mu := \frac{p-r}{\sqrt{r(1-r)}} 
\label{eq:def-mu}
\end{equation}
that can be understood as the signal-to-noise ratio of model $\cP$.
We have the following theorem.

\begin{theorem}[Detection lower bound]
\label{thm:detect-lower}
Consider Problem~\ref{prob:detect}.
Fix any constant $\delta \in (0, 0.1)$.
No polynomial $f \in \R[A]_{\le D}$ strongly separates $\cP$ and $\cQ$ in the sense of Definition~\ref{def:strong-separation}, if
\begin{equation}
n^3 \tau^4 \mu^6 \le n^{-\delta}, \qquad
\mu = \tilde O(1), \qquad
D = o \bigg( \Big(\frac{\log n}{\log \log n}\Big)^2 \bigg).
\label{eq:detect-lower-cond}
\end{equation}
\end{theorem}

To clarify, $\tilde{O}(1)$ in \eqref{eq:detect-lower-cond} does not stand for a specific bound but rather allows $\mu = \mu_n$ to be \emph{any} sequence that scales as $\tilde{O}(1)$, and similarly for the condition on $D$.
In addition, to ease the presentation, we have assumed the conditions in \eqref{eq:detect-lower-cond} that are stronger than what is required by the proof:
It suffices to assume $n^3 \tau^4 \mu^6 \le n^{-o(1)}$ for an appropriately defined $o(1)$ quantity, and the degree $D$ can be polylogarithmic in $n$ or even $n^{o(1)}$ if $\mu$ is sufficiently small.

The proof of the above theorem is deferred to Section~\ref{sec:detect-lower}.
We now provide a proof sketch.
To show that no polynomial of degree at most $D$ strongly separates $\cP$ and $\cQ$, it suffices to prove that the ``advantage''
$$
\Adv_{\le D} 
:= \sup_{\substack{f \in \R[A]_{\le D}, \\ \E_{\cQ}[f(A)^2] \ne 0}} \frac{\E_{\cP}[f(A)]}{\sqrt{\E_{\cQ}[f(A)^2]}} 
$$
is $O(1)$; see~\cite[Proposition~6.2]{fp}. 
Furthermore, it is known \cite[Section~2.3]{sam-thesis} that
\begin{equation}
\Adv_{\le D}^2 = \sum_{\alpha \subseteq \binom{[n]}{2} \,:\, |\alpha| \le D} \big( \E_{\cP} [ \phi_\alpha(A) ] \big)^2 .
\label{eq:adv-sq-sum}
\end{equation}
The rest of the proof consists in controlling all the summands in \eqref{eq:adv-sq-sum}, which is done in Section~\ref{sec:detect-lower}.
This eventually leads to Proposition~\ref{prop:adv-bound}, from which Theorem~\ref{thm:detect-lower} easily follows.

To further clarify the intuition behind the sum in \eqref{eq:adv-sq-sum}, for each subgraph $\alpha \subseteq \binom{[n]}{2}$, we can understand the quantity $\E_{\cP} [ \phi_\alpha(A) ]$ as the ``power'' of the statistic $\phi_\alpha(A)$ in distinguishing $\cP$ from $\cQ$.
The lower bound requires that the total distinguishing power, as a sum of $(\E_{\cP} [ \phi_\alpha(A) ])^2$ over all low-degree $\alpha$, is bounded.
% Informally, each summand $\big( \E_{\cP} [ \phi_\alpha(A) ] \big)^2$ can be understood as the ``signal'' provided by the subgraph $\alpha$ for the detection problem.
% The bulk of the proof consists in bounding the signal corresponding each $\alpha$.
% This is done in Section~\ref{sec:detect-lower}, 
On the other hand, if $\E_{\cP} [ \phi_\alpha(A) ]$ is large for a particular choice of $\alpha$, then the corresponding statistic $\phi_\alpha(A)$ can be used for testing between $\cP$ and $\cQ$.
A careful study of $\E[\phi_\alpha(A)]$ in Proposition~\ref{prop:adv-bound} suggests that the bottleneck case is when the graph $\alpha$ is a triangle.
Therefore, it is natural to consider signed triangles for the upper bound.

\subsection{Upper bound}
While signed triangles have been analyzed for random geometric graphs in previous works such as \cite{bubeck2016testing,brennan2020phase,liu2021phase,liu2022testing}, none of these results apply in our case.
For example, the setup closest to ours can be found in \cite{liu2021phase}, where high-dimensional random geometric graphs are studied but the probability $p$ has to be fixed.
Therefore, we present a self-contained analysis of the signed triangle statistic
\begin{equation}
S_3(A) := \sum_{H \in \binom{[n]}{3}} \prod_{(i,j) \in \binom{H}{2}} \bar{A}_{ij} .
\label{eq:signed-triangle}
\end{equation}
Note that if $\bar A_{ij}$ were replaced by $A_{ij}$ in the above definition, then $S_3(A)$ would be the number of triangles in the graph $A$. 
Hence $S_3(A)$ is a standardized version of triangle count. 

\begin{theorem}[Detection upper bound]
\label{thm:detect-upper}
Consider Problem~\ref{prob:detect}.
Suppose that $p \ge C q$ for a constant $C > 1$.
The degree-$3$ polynomial $S_3(A)$ defined in \eqref{eq:signed-triangle} strongly separates $\cP$ and $\cQ$ in the sense of Definition~\ref{def:strong-separation}, if 
\begin{equation}
n^3 \tau^4 p^6/r^3 = \omega(1) , \qquad 
n^3 \tau^2 p^3 = \omega(1) .
\label{eq:detect-upper-cond}
\end{equation}
\end{theorem}

The proof of the theorem is deferred to Section~\ref{sec:detect-upper}. 
In short, we control the two expectations $\E_\cP[f(A)]$, $\E_\cQ[f(A)]$ and the two variances $\Var_\cP[f(A)]$, $\Var_\cQ[f(A)]$ in Propositions~\ref{prop:q-exp-var}, \ref{prop:p-exp}, and~\ref{prop:p-var}, which together result in Theorem~\ref{thm:detect-upper}.
We have again chosen simplicity over generality for the statement of the above theorem by assuming $p \ge Cq$ for $C>1$.
A more general condition can be obtained from a refined comparison between the bounds in Propositions~\ref{prop:q-exp-var} and \ref{prop:p-var}.
The two conditions in \eqref{eq:detect-upper-cond} can be interpreted as follows.
First, as we see in the proof, $\mu = \Theta(p/r^{1/2})$, so the first condition in \eqref{eq:detect-upper-cond} matches the first condition in \eqref{eq:detect-lower-cond} up to an $n^{\delta}$ factor; they together give the detection threshold stated in Theorem~\ref{thm:main-informal}.
Next, for any three vertices, the probability that they are neighbors in the planted cycle and form a triangle in $A$ is $\Theta(\tau^2 p^3)$; as a result, there are $\Theta(n^3 \tau^2 p^3)$ triangles in the planted cycle on average.
Therefore, the second condition in \eqref{eq:detect-upper-cond} is a minimal condition guaranteeing the existence of triangles in the planted cycle in the first place.
Further, note that if $p$ and $q$ are of the same order, then $n^3 \tau^2 p^3 = \Omega( n^3 \tau^4 p^6 / r^3 )$, so the second condition in \eqref{eq:detect-upper-cond} is subsumed by the first condition.

\section{The recovery problem} 
\label{sec:recover}

Similar to the previous section, we start with the low-degree lower bound, whose proof suggests an optimal efficient algorithm.
Then we analyze the algorithm to establish the matching upper bound.
The optimal statistic for recovery turns out to be a signed count of self-avoiding walks between vertices $1$ and $2$, a statistic that has been used for related problems such as community detection in \cite{hopkins2017efficient} and spiked matrix models in \cite{ding2020estimating}.

\subsection{Lower bound}
A general strategy for proving low-degree lower bounds for estimation problems was proposed by \cite{schramm2020computational}.
We provide a lower bound in Proposition~\ref{prop:lower-bound-estimation-general} that extends the one in \cite[Section~3.5]{schramm2020computational} for the planted clique problem.
Let us start with a general recovery problem with binary observations.

\begin{definition}
\label{def:general-binary-model}
For an integer $N \ge 1$, let $B_1, \dots, B_N$ be i.i.d.\ $\Bern(q)$ variables. 
Consider a latent random subset $W \subseteq [N]$ from an arbitrary prior over subsets of $[N]$. 
Conditional on $W$, we define the observation $A \in \R^N$ as follows.
If $i \notin W$, then let $A_i := B_i$. 
If $i \in W$, then sample an independent $A_i \sim \Bern(p)$.
\end{definition}

Given $A$ from the above model, we aim to estimate $\chi := \bbone\{1 \in W\}$.
For a positive integer $D$, define
$$
\Corr_{\le D} 
:= \sup_{\substack{f \in \R[A]_{\le D}, \\ \E[f(A)^2] \ne 0}} \frac{\E[f(A) \cdot \chi]}{\sqrt{\E[f(A)^2]}} 
$$
as in \eqref{eq:corr-d}.
Let
\begin{equation}
\lambda := \frac{p-q}{\sqrt{q(1-q)}} ,
\label{eq:def-lambda}
\end{equation}
which is a signal-to-noise ratio analogous to $\mu$ in \eqref{eq:def-mu} for the detection problem (here in the recovery problem, model $\cQ$ is irrelevant, so $r$ is replaced by $q$ in the definition of $\lambda$).
The following result is proved in Section~\ref{sec:recover-lower}.

\begin{proposition}
\label{prop:lower-bound-estimation-general}
Assume the model in Definition~\ref{def:general-binary-model}. 
For $\beta \subseteq \alpha \subseteq [N]$, let 
$$
P_{\alpha \beta} := \p\{\alpha \setminus W = \beta\} .
$$
Suppose $P_{\alpha \alpha} > 0$ for all $\alpha \subseteq [N]$. 
Then we have 
$$
\Corr_{\le D}^2 \le \sum_{\alpha \subseteq [N] \,:\, |\alpha| \le D} \rho_\alpha^2 \, \lambda^{2 |\alpha|} ,
$$
where $\rho_\alpha$ is defined recursively by
$\rho_\varnothing := \p\{1 \in W\}$
and
\begin{align*}
\rho_\alpha 
:= \frac{1}{P_{\alpha \alpha}} \bigg( \p\{\alpha \cup \{1\} \subseteq W\} - \sum_{\beta \subsetneq \alpha} \rho_\beta \, P_{\alpha \beta} \bigg) .
\end{align*}
\end{proposition}

We now return to the problem of planted dense cycle and present the following result.

\begin{theorem}[Recovery lower bound]
\label{thm:recover-lower}
Consider Problem~\ref{prob:recover}.
Fix any constant $\delta \in (0, 0.1)$.
No polynomial $f \in \R[A]_{\le D}$ weakly recovers $\chi = \bbone\{\dist(z_1, z_2) \le \tau/2\}$ given $A \sim \cP$ in the sense of Definition~\ref{def:weak-recovery}, if
\begin{equation}
n \tau^2 \lambda^2 \le n^{-\delta}, \qquad
\lambda = O(1), \qquad
D = o \bigg( \Big(\frac{\log n}{\log \log n}\Big)^2 \bigg), \qquad
\tau D^4 \le 0.1.
\label{eq:recover-lower-cond}
\end{equation}
\end{theorem}

Let us discuss the conditions in \eqref{eq:recover-lower-cond}, which are analogous to those in \eqref{eq:detect-lower-cond}.
First, the main condition is $n \tau^2 \lambda^2 \le n^{-\delta}$, which can be weakened to $n \tau^2 \lambda^2 \le n^{-o(1)}$ for an appropriately defined $o(1)$ quantity by a closer inspection of the above proof.
Also, the degree $D$ can be polylogarithmic in $n$ or even $n^{o(1)}$ if $\lambda$ is sufficiently small.
Finally, the technical condition $\tau D^4 \le 0.1$ is inactive if $n \lambda^2 \ge 1$; even if it is active, the condition is mild because the interesting regime of small-world networks is where the bandwidth $n \tau$ is much smaller than the total number of vertices $n$.

Theorem~\ref{thm:recover-lower} is proved in Section~\ref{sec:recover-lower} and we now provide a sketch.
To apply Proposition~\ref{prop:lower-bound-estimation-general}, we note that model~$\cP$ in Definition~\ref{def:mod-p} is a special case of the model in Definition~\ref{def:general-binary-model}.
Namely, let $N = \binom{n}{2}$, use an index pair $(i,j) \in \binom{[n]}{2}$ instead of a single index, and let
\begin{equation}
W = \left\{ (i,j) \in \binom{[n]}{2} \,:\, \dist(z_i,z_j) \le \tau/2 \right\}. 
\label{eq:W-meaning}
\end{equation}
In addition, we have
$$
\chi = \bbone\{\dist(z_{1},z_{2}) \le \tau/2\} = \bbone\{(1, 2) \in W\} .
$$
Proposition~\ref{prop:lower-bound-estimation-general} then implies that 
\begin{equation}
\Corr_{\le D}^2 \le \sum_{\alpha \subseteq \binom{[n]}{2} \,:\, |\alpha| \le D} \rho_\alpha^2 \, \lambda^{2 |\alpha|} ,
\label{eq:corr-rho-bound}
\end{equation}
where $\rho_\alpha$ is defined recursively by
$\rho_\varnothing = \p\{\dist(z_{1},z_{2}) \le \tau/2\} = \tau$, 
and
\begin{equation}
\rho_\alpha  
= \frac{1}{P_{\alpha \alpha}} \bigg( \p\{\alpha \cup \{(1, 2)\} \subseteq W\} - \sum_{\beta \subsetneq \alpha} \rho_\beta \, P_{\alpha \beta} \bigg) .
\label{eq:rho-recursion}
\end{equation}
Then the bulk of the proof consists in bounding $\rho_\alpha^2$ for each $\alpha$ using the above recursion.
This is done in Section~\ref{sec:recover-lower}, eventually leading to the bounds on $\Corr_{\le D}^2$ in Proposition~\ref{prop:corr-bound}.
Theorem~\ref{thm:recover-lower} then follows as a consequence.

% The key step of the above proof is controlling the quantities $\rho_\alpha$ in \eqref{eq:corr-rho-bound}, which is deferred to Proposition~\ref{prop:corr-bound}.
The recursive definition \eqref{eq:rho-recursion} is similar to that for joint cumulants of the random variables $\chi$ and $(A_{ij})_{(i,j) \in \alpha}$ (see \cite{schramm2020computational}).
Intuitively, for each $\alpha \subseteq \binom{[n]}{2}$, the cumulant-like quantity $\rho_\alpha$ measures the amount of ``information'' $(A_{ij})_{(i,j) \in \alpha}$ contains about the estimand $\chi$.
The above lower bound controls the total amount of information that all subgraphs with at most $D$ edges have about $\chi$.
On the other hand, if $\rho_\alpha$ is large for a particular choice of $\alpha$, then the corresponding subgraph $(A_{ij})_{(i,j) \in \alpha}$ may be useful for recovering $\chi$.
The analysis of $\rho_\alpha$ in Proposition~\ref{prop:corr-bound} turns out to suggest that we should consider self-avoiding walks between vertices $1$ and $2$, which we study in the next subsection for the upper bound.

\subsection{Upper bound}

Similar to $\bar{A}$ in \eqref{eq:def-a-bar}, we consider a standardized version $\tilde A$ of the observed graph, defined by
\begin{equation}
\tilde A_{ij} := \frac{A_{ij} - q}{\sqrt{q(1-q)}} 
\label{eq:def-a-tilde}
\end{equation}
for $(i,j) \in \binom{[n]}{2}$.
Compared to \eqref{eq:def-a-bar}, the parameter $r$ is replaced by $q$ in \eqref{eq:def-a-tilde} because model $\cQ$ is irrelevant for the recovery problem.
Moreover, for $\alpha \subset \binom{[n]}{2}$, define
\begin{equation}
\tilde A_\alpha := \prod_{(i,j) \in \alpha} \tilde A_{ij} .
\label{eq:def-a-alpha}
\end{equation}

As discussed above, the proof of the recovery lower bound suggests that self-avoiding walks between vertices $1$ and $2$ are informative about $\chi = \bbone\{\dist(z_1,z_2) \le \tau/2\}$, which motivates us to consider the following.
Fix an integer $\ell \ge 1$.
Let $\SAW$ be the set of all length-$(\ell+1)$ self-avoiding walks from vertex $1$ and to vertex $2$, i.e.,
\begin{equation}
\SAW := \Big\{ \{(1,i_1), (i_1,i_2), (i_2, i_3), \ldots , (i_{\ell-1}, i_\ell), (i_\ell, 2)\} : i_1, \dots, i_\ell, 1, 2 \text{ are all distinct} \Big\} .
\label{eq:def-saw}
\end{equation}
Define the signed count of $\SAW$ in the observed graph $A$ as
\begin{equation}
T(A) := \sum_{\alpha\in\SAW}\tilde{A}_\alpha = \sum_{i_1\neq \cdots \neq i_\ell \ne 1 \ne 2} \tilde{A}_{1i_1}\tilde{A}_{i_1i_2} \tilde A_{i_2 i_3} \cdots \tilde{A}_{i_{\ell-1} i_\ell} \tilde{A}_{i_\ell 2}.
\label{eq:def-stat-t}
\end{equation}
As discussed above, this statistic has appeared in, e.g., \cite{hopkins2017efficient} for community detection.
The following theorem shows that the statistic $T(A)$ achieves weak recovery of $\chi$, and its proof can be found in Section~\ref{sec:recover-upper-proof}.

\begin{theorem}[Recovery upper bound]
\label{thm:recover-upper}
Consider Problem~\ref{prob:recover}.
Suppose that $C q \le p \le C' q$ for constants $C' > C > 1$.
For any constant $\delta \in (0,0.1)$, fix an integer $\ell > 1/\delta$.
The degree-$(\ell+1)$ polynomial $T(A)$ defined in \eqref{eq:def-stat-t} weakly recovers $\chi = \bbone\{\dist(z_1,z_2) \le \tau/2\}$ given $A \sim \cP$ in the sense of Definition~\ref{def:weak-recovery}, if 
\begin{equation}
n \tau^2 p \ge n^{\delta} , \qquad 
\tau = o(1) .
\label{eq:recover-upper-cond}
\end{equation}
\end{theorem}

If $p$ and $q$ are of the same order, then we have $\lambda = \frac{p-q}{\sqrt{q(1-q)}} = \Theta(p^{1/2})$.
Therefore, the main condition $n \tau^2 p \ge n^{\delta}$ in \eqref{eq:recover-upper-cond} matches the first condition in \eqref{eq:recover-lower-cond} up to an $n^{\delta}$ factor; they together give the recovery threshold stated in Theorem~\ref{thm:main-informal}.
We can also obtain a more general condition for the upper bound using Propositions~\ref{prop:recovery-expectation} and~\ref{prop:recovery-variance}, but the condition is not tight in the regime where $p/q \ge n^c$ for a constant $c>0$.
Proving a tight condition requires more technical work beyond the scope of this paper.
Moreover, as we have explained, the condition $\tau = o(1)$ in \eqref{eq:recover-upper-cond} is natural because the bandwidth is usually much smaller than the total number of vertices.

\medskip

We have focused on weak recovery in the sense of Definition~\ref{def:weak-recovery} and established the detection-recovery gap in the framework of low-degree polynomials.
Let us now consider the more practical problem of exactly recovering the indicator $\chi = \bbone\{\dist(z_1, z_2) \le \tau/2\}$ with high probability using a polynomial-time algorithm.
Towards this end, fix a quantity $\epsilon \in (0, \tau/2)$ and define
\begin{equation}
\kappa = \kappa(\eps) := \frac{\E\kr{T\cond \dfrak(z_1,z_2)=\frac{\tau}{2}} + \E\kr{T\cond \dfrak(z_1,z_2)=\frac{\tau}{2}+\epsilon}}{2}.
\label{eq:def-kappa}
\end{equation}
By Proposition~\ref{prop:recovery-expectation}, the quantity $\kappa$ can be computed explicitly.
We then threshold the statistic $T(A)$ in \eqref{eq:def-stat-t} at $\kappa$ to obtain the estimator 
$$
\hat \chi := \bbone\{T(A) \ge \kappa\}.
$$
The following result is a consequence of our analysis of the statistic $T(A)$, and its proof is deferred to the end of Section~\ref{sec:recover-upper-proof}.

\begin{theorem}
\label{thm:exact-recovery}
In the setting of Theorem~\ref{thm:recover-upper}, we additionally assume $\ell > 3/\delta$ and set $\eps := \tau n^{-\delta/4}$.
Then the estimator $\hat \chi = \bbone\{T(A) \ge \kappa\}$ of $\chi = \bbone\{\dist(z_1,z_2) \le \tau/2\}$ satisfies
$$
\E\kr{(\hat \chi - \chi)^2} = \Pb\br{\hat\chi \neq \chi} \le C_\ell \, \tau n^{-\delta/2}
$$
for a constant $C_\ell > 0$ depending only on $\ell$.
\end{theorem}

\noindent
Since $\dist(z_1,z_2) \le \tau/2$ with probability $\tau$, a trivial estimator $\tilde \chi \equiv 0$ makes an error with probability $\tau$.
Therefore, the above error probability $O(\tau n^{-\delta/2})$ is small.

An immediate consequence of the above result is that we can estimate the underlying random geometric graph consistently.
To be more precise, we denote the adjacency matrix of the geometric graph by $X \in \{0,1\}^{n \times n}$, which is defined by $X_{ij} := \bbone\{\dist(z_i,z_j) \le \tau/2\}$.

\begin{corollary}
\label{cor:rgg}
In the setting of Theorem~\ref{thm:exact-recovery}, there is an estimator $\hat X \in \{0,1\}^{n \times n}$ of the random geometric graph $X \in \{0,1\}^{n \times n}$ such that
$$
\E[\|\hat X - X\|_F^2] \le C_\ell \, \tau n^{2 - \delta/2}
$$
for a constant $C_\ell > 0$ depending only on $\ell$.
\end{corollary}

\noindent
This result follows immediately from Theorem~\ref{thm:exact-recovery}, because by symmetry, it suffices to estimate each edge $X_{ij} = \bbone\{\dist(z_i,z_j) \le \tau/2\}$ in the same way as we did for $(i,j) = (1,2)$.

\section{Additional proofs}

\subsection{Detection lower bound}
\label{sec:detect-lower}

In this subsection, we establish Proposition~\ref{prop:adv-bound} which leads to Theorem~\ref{thm:detect-lower}.
Recall the models $\cP$ and $\cQ$ in Definitions~\ref{def:mod-p} and~\ref{def:mod-q} respectively.
For $\alpha \subseteq \binom{[n]}{2}$, define
$$
\eta(z; \alpha) := \big| \big\{(i,j) \in \alpha : \dist(z_i,z_j) \le \tau/2 \big\} \big| .
$$

\begin{lemma}
\label{lem:connected-alpha-bd}
For $\alpha \subseteq \binom{[n]}{2}$, let $v := |V(\alpha)|$ and suppose that the graph $\alpha$ is connected. Then we have 
$$
\E_z \Big[ \frac{1}{\tau^{\eta(z; \alpha)}} \Big] \le v^{2v} \tau^{v-|\alpha|-1} .
$$
\end{lemma}

\begin{proof}
Suppose that $V(\alpha) = \{i_1, \dots, i_v\}$. 
For any realization of $z$, there is a unique partition $B_1 \sqcup \cdots \sqcup B_m$ of $\{i_1, \dots, i_v\}$ such that the following two conditions hold: 
\begin{enumerate}
\item \label{cond:1}
For any distinct $j, j' \in [m]$ and any $\ell \in B_j$ and $\ell' \in B_{j'}$, we have $\dist(z_\ell,z_{\ell'}) > \tau/2$;

\item \label{cond:2}
For any $j \in [m]$, $B_j$ cannot be partitioned into two sub-blocks satisfying Condition 1. 
\end{enumerate}
In other words, we partition $z_{i_1}, \dots, z_{i_v}$ into blocks so that the distance between two consecutive points in the same block is at most $\tau/2$.

Now fix a partition $\{i_1, \dots, i_v\} = B_1 \sqcup \cdots \sqcup B_m$. 
We claim that 
\begin{align}
\p_z \big\{ \text{Conditions~\ref{cond:1} and~\ref{cond:2} are satisfied for } B_1, \dots, B_m \big\} \le (v \tau)^{v - m}.
\label{eq:block-bd}
\end{align}
To prove \eqref{eq:block-bd}, it suffices to use Condition~\ref{cond:2}.
Fix $\ell_j \in B_j$ for $j \in [m]$. 
By Condition~\ref{cond:2}, for any $j \in [m]$ and $\ell \in B_j$, we have $\dist(z_\ell,z_{\ell_j}) \le |B_j| \, \tau/2 \le v \tau/2$. 
For any realization of $z_{\ell_1}, \dots, z_{\ell_m}$, it holds that
\begin{align*}
\p_z \big\{ \dist(z_\ell,z_{\ell_j}) \le v \tau/2 \text{ for all } \ell \in B_j \text{ and all } j \in [m] \mid z_{\ell_1}, \dots, z_{\ell_m} \big\}
\le \prod_{j=1}^m (v \tau)^{|B_j|-1} 
= (v \tau)^{v - m}.
\end{align*}
Then \eqref{eq:block-bd} follows. 

Since the graph $\alpha$ is connected, there are at least $m-1$ edges between vertices $\ell \in B_j$ and $\ell' \in B_{j'}$ for distinct $j, j' \in [m]$. 
If Condition~\ref{cond:1} is satisfied, then we have
\begin{align}
\eta(z; \alpha) 
= \big| \big\{(\ell,\ell') \in \alpha : \dist(z_\ell, z_{\ell'}) \le \tau/2 \big\} \big| 
\le |\alpha| - (m-1) .
\label{eq:cond-1-result}
\end{align}

Combining \eqref{eq:block-bd} and \eqref{eq:cond-1-result}, we obtain
\begin{align*}
\E_z \Big[ \frac{1}{\tau^{\eta(z; \alpha)}} \Big] 
\le \sum_{B_1 \sqcup \cdots \sqcup B_m = \{i_1, \dots, i_v\}} 
(v \tau)^{v-m} \cdot \frac{1}{\tau^{|\alpha| - (m-1)}} 
\le \sum_{B_1 \sqcup \cdots \sqcup B_m = \{i_1, \dots, i_v\}}  v^v \tau^{v-|\alpha|-1} .
\end{align*}
Finally, bound the number of partitions by $v^v$. 
\end{proof}

\begin{lemma}
\label{lem:general-alpha-bd}
For $\alpha \subseteq \binom{[n]}{2}$, let $v := |V(\alpha)|$, and let $m$ be the number of connected components of the graph $\alpha$.
Then we have 
$$
\E_z \Big[ \frac{1}{\tau^{\eta(z; \alpha)}} \Big] 
\le v^{2v} \tau^{v-|\alpha|-m} .
$$
\end{lemma}

\begin{proof}
Let $\alpha^{(1)} \sqcup \cdots \sqcup \alpha^{(m)}$ denote the partition of $\alpha$ into connected components, and note that
$$
|\alpha| = \sum_{i=1}^m |\alpha^{(i)}|, \qquad
v = |V(\alpha)| = \sum_{i=1}^m |V(\alpha^{(i)})|, 
\qquad 
\eta(z; \alpha) = \sum_{i=1}^m \eta(z; \alpha^{(i)}).
$$
Lemma~\ref{lem:connected-alpha-bd} shows that for each connected component $\alpha^{(i)}$, we have
\begin{align*}
\E_z \Big[ \frac{1}{\tau^{\eta(z; \alpha^{(i)})}} \Big] 
\le |V(\alpha^{(i)})|^{2 |V(\alpha^{(i)})|} \, \tau^{|V(\alpha^{(i)})| - |\alpha^{(i)}| -1}
\le v^{2 |V(\alpha^{(i)})|} \, \tau^{|V(\alpha^{(i)})| - |\alpha^{(i)}| -1} .
\end{align*}
Crucially, the random variables $\eta(z; \alpha^{(1)}), \dots, \eta(z; \alpha^{(m)})$ are independent because the connected components have mutually disjoint vertex sets and thus involve independent collections of latent variables $z_j$.
We conclude that
\begin{align*}
\E_z \Big[ \frac{1}{\tau^{\eta(z; \alpha)}} \Big]
= \prod_{i=1}^m \E_z \Big[ \frac{1}{\tau^{\eta(z; \alpha^{(i)})}} \Big] 
\le \prod_{i=1}^m v^{2 |V(\alpha^{(i)})|} \, \tau^{|V(\alpha^{(i)})| - |\alpha^{(i)}| -1}
= v^{2v} \tau^{v-|\alpha|-m} .
\end{align*}
\end{proof}

\begin{lemma}
\label{lem:phi-expectation-bd}
For $\alpha \subseteq \binom{[n]}{2}$, let $v := |V(\alpha)|$, and let $m$ be the number of connected components of the graph $\alpha$.
Recall that $\mu = \frac{p-r}{\sqrt{r(1-r)}}$.
Then we have 
\begin{equation*}
\big| \E_{\cP} [ \phi_\alpha(A) ] \big| \le \Big(\frac{\mu}{1-\tau}\Big)^{|\alpha|} v^{2v} \tau^{v-m} .
\end{equation*}
\end{lemma}

\begin{proof}
We have 
\begin{align*}
\E_{\cP} [ \phi_\alpha(A) ] 
&= \frac{1}{(r(1-r))^{|\alpha|/2}} \E_{\cP} \bigg[ \prod_{(i,j) \in \alpha} (A_{ij} - r) \bigg] \\
&= \frac{1}{(r(1-r))^{|\alpha|/2}} \E_z \bigg[ \prod_{(i,j) \in \alpha} \E[A_{ij} - r \mid z] \bigg] \\
&= \frac{1}{(r(1-r))^{|\alpha|/2}} \E_z \Big[ (p-r)^{\eta(z; \alpha)} (q-r)^{|\alpha| - \eta(z; \alpha)} \Big] .
\end{align*}
Recall that $r = \tau p + (1-\tau) q$ so that $\frac{p-r}{q-r} = \frac{\tau-1}{\tau}$, and $\mu = \frac{p-r}{\sqrt{r(1-r)}}$.
It follows that
$$
\E_{\cP} [ \phi_\alpha(A) ] 
= \bigg(\frac{\mu \tau}{\tau-1}\bigg)^{|\alpha|} \E_z \bigg[ \Big(\frac{\tau-1}{\tau}\Big)^{\eta(z; \alpha)} \bigg] .
$$
By Lemma~\ref{lem:general-alpha-bd}, we obtain
\begin{align*}
\big| \E_{\cP} [ \phi_\alpha(A) ] \big| 
\le \bigg(\frac{\mu \tau}{1-\tau}\bigg)^{|\alpha|} \, \E_z \Big[ \frac{1}{\tau^{\eta(z; \alpha)}} \Big] 
\le \bigg(\frac{\mu \tau}{1-\tau}\bigg)^{|\alpha|} \, v^{2v} \tau^{v-|\alpha|-m} ,
\end{align*}
finishing the proof.
\end{proof}

\begin{lemma}
\label{lem:phi-zero}
If the graph $\alpha$ has a dangling edge, i.e., an edge $(i,j)$ where $i$ is connected only to $j$ in $\alpha$, 
then $\E_{\cP}[\phi_\alpha(A)] = 0$.
\end{lemma}

\begin{proof}
Let $z_{-i} = (z_1, \dots, z_{i-1}, z_{i+1}, \dots, z_n)$.
Since $i$ is connected only to $j$ in $\alpha$, conditional on $z_{-i}$, the edges $\left\{A_{i'j'} : (i',j') \in \alpha \setminus \{(i,j)\} \right\}$ are mutually independent; further, they are independent from $z_i$ and $A_{ij}$. 
It follows that 
\begin{align*}
\E_{\cP} [ \phi_\alpha(A) ]
= \E_{z_{-i}} \bigg[  \prod_{(i',j') \in \alpha} \E_{\cP} [ \bar A_{i'j'} \mid z_{-i} ] \bigg] = 0 ,
\end{align*}
because 
$$
\E_{\cP}[\bar A_{ij} \mid z_{-i}] 
= \frac{\E_{\cP}[A_{ij} - r \mid z_j]}{\sqrt{r(1-r)}} 
= \frac{\tau (p-r) + (1-\tau) (q-r)}{\sqrt{r(1-r)}} 
= 0
$$
by the definition $r := \tau p + (1-\tau) q$.
\end{proof}

\begin{proposition}
\label{prop:adv-bound}
Recall \eqref{eq:adv-sq-sum}.
We have $\Adv_{\le D}^2 \le 2$ in either of the following situations:
\begin{itemize}
\item
$2 n^3 \tau^4 (2D)^{13} R^{6 \sqrt{D}} \le 1/2$, where $R := \max \big\{ 2D \frac{\mu}{1-\tau}, 1 \big\}$;

\item
$L \big(2D \frac{\mu}{1-\tau}\big)^2 \le 1/2$
and
$n^3 \tau^4 (2D)^{19} \big(\frac{\mu}{1-\tau}\big)^6 \le 1/2$, 
where $L := \max\{ n \tau^2 (2D)^4, 1 \}$.
\end{itemize}
\end{proposition}

\begin{proof}
To ease the notation, we consider $D$ such that $\sqrt{D}/3$ is an integer; the proof can be easily adapted to the general case by using floors $\lfloor \cdot \rfloor$ or ceilings $\lceil \cdot \rceil$. 

We start with \eqref{eq:adv-sq-sum} which states
$$
\Adv_{\le D}^2 = \sum_{\alpha \subseteq \binom{[n]}{2} \,:\, |\alpha| \le D} \big( \E_{\cP} [ \phi_\alpha(A) ] \big)^2 .
$$
Recall that $\phi_\varnothing \equiv 1$. 
Let $\mathsf{c}(\alpha)$ denote the number of connected components of $\alpha$. 
If any connected component of $\alpha$ has less than three edges, then it must contains a dangling edge;
if the number of vertices of $\alpha$ exceeds the number of edges, then $\alpha$ also has a dangling edge. 
Therefore, Lemma~\ref{lem:phi-zero} shows that $\E_{\cP} [ \phi_\alpha(A) ] = 0$ if $|\alpha| < 3 \, \mathsf{c}(\alpha)$ or $|V(\alpha)| > |\alpha|$.
Then, by Lemma~\ref{lem:phi-expectation-bd}, we obtain
\begin{align*}
\Adv_{\le D}^2 
\le 1 + \sum_{m=1}^{ D/3 } \sum_{\ell = 3m}^D \sum_{v = 3m}^{D} \sum_{\substack{\alpha \subseteq \binom{[n]}{2} \,:\, \, \mathsf{c}(\alpha) = m , \\ |\alpha| = \ell, \, |V(\alpha)| = v}} 
\Big(\frac{\mu}{1-\tau}\Big)^{2 \ell} v^{4 v} \tau^{2v-2m} \cdot \bbone\{v \le \ell \le v^2/2\} .
\end{align*}
There are at most $\binom{n}{v} \, \Big[ \binom{\binom{v}{2}}{\ell} \land 2^{\binom{v}{2}} \Big]$ graphs $\alpha$ with $|V(\alpha)| = v$ and $|\alpha| = \ell$. 
By the inequalities $\binom{n}{v} \le n^v$, $\binom{\binom{v}{2}}{\ell} \le v^{2 \ell} \le (2D)^{2 \ell}$, and $2^{\binom{v}{2}} \le 2^{v^2}$, we have
\begin{align}
&\Adv_{\le D}^2 
\le 1 + \sum_{m=1}^{ D/3 } \sum_{\ell = 3m}^D \sum_{v = 3m}^{D} \notag \\ 
&\qquad \qquad \Big[(2D)^{2\ell} \land 2^{v^2}\Big] \cdot \Big(\frac{\mu}{1-\tau}\Big)^{2 \ell} \big( n \tau^2 (2D)^4 \big)^v \tau^{-2m} \cdot \bbone\{v \le \ell \le v^2/2\} .
\label{eq:adv-bound-inter}
\end{align}
Let us consider two cases.
\paragraph{Case 1:} 
We bound \eqref{eq:adv-bound-inter} by splitting it into the following terms according the value of $v$:
\begin{subequations}
\begin{align}
\Adv_{\le D}^2 
\le 1 &+ \sum_{m=1}^{ D/3 } \sum_{v = 3m}^{ \sqrt{D} } \sum_{\ell = 3m}^D 2^{v^2} \Big(\frac{\mu}{1-\tau}\Big)^{2 \ell} \big( n \tau^2 (2D)^4 \big)^v \tau^{-2m} \cdot \bbone\{\ell \le v^2/2\} \label{eq:adv-bd-1} \\
&+ \sum_{m=1}^{ D/3 } \sum_{v =  \sqrt{D}  \lor 3m}^{D} \sum_{\ell = 3m}^D (2D)^{2\ell} \Big(\frac{\mu}{1-\tau}\Big)^{2 \ell} \big( n \tau^2 (2D)^4 \big)^v \tau^{-2m} . \label{eq:adv-bd-2}
\end{align}
\end{subequations}
We then bound \eqref{eq:adv-bd-1} and \eqref{eq:adv-bd-2} respectively. 
Recall that $R := \max \big\{ 2D \frac{\mu}{1-\tau}, 1 \big\}$. 
By the assumption $2 n^3 \tau^4 (2D)^{13} R^{6 \sqrt{D}} \le 1/2$, we have
$$
n \tau^2 (2D)^4 R^{\sqrt{D}} \le 1/2 , \qquad
n^3 \tau^4 (2D)^{12} R^{3 \sqrt{D}} \le 1/2 .
$$
For $v \le \sqrt{D}$, we have $2 \ell \le v^2 \le \sqrt{D} \, v$, so the sum in \eqref{eq:adv-bd-1} is bounded by
\begin{align*}
\sum_{m=1}^{ D/3 } \sum_{v = 3m}^{ \sqrt{D} } D \, R^{\sqrt{D} \, v} \big( n \tau^2 (2D)^4 \big)^v \tau^{-2m}
&\le 2 D \sum_{m=1}^{ D/3 } \big( n \tau^2 (2D)^4 R^{\sqrt{D}} \big)^{3m} \tau^{-2m} \\
&\le 4 D \, n^3 \tau^4 (2D)^{12} R^{3 \sqrt{D}} \le 1/2
\end{align*}
by the assumption $2 n^3 \tau^4 (2D)^{13} R^{6 \sqrt{D}} \le 1/2$.
Next, using 
$$
n \tau^2 (2D)^4 \le 1/2, \qquad
n^3 \tau^4 (2D)^{12} \le 1/2 ,
$$
we see that the sum in \eqref{eq:adv-bd-2} is bounded by
\begin{align*}
&\sum_{m=1}^{ D/3 } D R^{2 D} \big( n \tau^2 (2D)^4 \big)^{ \sqrt{D}  \lor 3m} \tau^{-2m} \\
&\le \sum_{m=1}^{ \sqrt{D}/3 } D R^{2 D} \big( n \tau^2 (2D)^4 \big)^{ \sqrt{D} } \tau^{-2  \sqrt{D}/3 } + \sum_{m= \sqrt{D}/3 }^{ D/3 } D R^{2 D} \big( n^3 \tau^4 (2D)^{12} \big)^{m} \\
&\le D^{3/2} R^{2 D} \big( n^3 \tau^4 (2D)^{12} \big)^{ \sqrt{D}/3 } + D R^{2 D} \big( n^3 \tau^4 (2D)^{12} \big)^{ \sqrt{D}/3 } \\
&\le \Big( n^3 \tau^4 (2D)^{12} (2D)^{9/(2\sqrt{D})} R^{6 \sqrt{D}} \Big)^{\sqrt{D}/3} \le 1/2 
\end{align*}
by the assumption $2 n^3 \tau^4 (2D)^{13} R^{6 \sqrt{D}} \le 1/2$.
Combining the two terms yields $\Adv_{\le D}^2 \le 2$.

\paragraph{Case 2:} 
It follows from \eqref{eq:adv-bound-inter} that 
\begin{align*}
\Adv_{\le D}^2 
&\le 1 + \sum_{m=1}^{ D/3 } \sum_{\ell = 3m}^D \sum_{v = 3m}^{\ell} n^v (2D)^{2 \ell} \Big(\frac{\mu}{1-\tau}\Big)^{2 \ell} (2D)^{4 v} \tau^{2v-2m} \\
&= 1 + \sum_{m=1}^{ D/3 } \tau^{-2m} \sum_{\ell = 3m}^D \bigg( \Big(2D \frac{\mu}{1-\tau}\Big)^2 \bigg)^\ell \sum_{v = 3m}^{\ell} \big( n \tau^2 (2D)^4 \big)^v .
\end{align*}
Recall that $L := \max\{ n \tau^2 (2D)^4, 1 \}$, $L \big(2D \frac{\mu}{1-\tau}\big)^2 \le 1/2$, and $n^3 \tau^4 (2D)^{19} \big(\frac{\mu}{1-\tau}\big)^6 \le 1/2$. 
We have
\begin{align*}
\Adv_{\le D}^2 
&\le 1 + \sum_{m=1}^{ D/3 } \tau^{-2m} \sum_{\ell = 3m}^D \bigg( \Big(2D \frac{\mu}{1-\tau}\Big)^2 \bigg)^\ell (\ell - 3m + 1) \big( n \tau^2 (2D)^4 \big)^{3m} L^{\ell - 3m} \\
&\le 1 + D \sum_{m=1}^{ D/3 } \big( L^{-3}  n^3 \tau^4 (2D)^{12} \big)^m \sum_{\ell = 3m}^D \bigg( L \Big(2D \frac{\mu}{1-\tau}\Big)^2 \bigg)^\ell \\
&\le 1 + 2 D \sum_{m=1}^{ D/3 } \big( L^{-3} n^3 \tau^4 (2D)^{12} \big)^m \bigg( L \Big(2D \frac{\mu}{1-\tau}\Big)^2 \bigg)^{3m} \\
&= 1 + 2 D \sum_{m=1}^{ D/3 } \bigg( n^3 \tau^4 (2D)^{18} \Big(\frac{\mu}{1-\tau}\Big)^6 \bigg)^m \\
&\le 1 + 4 D n^3 \tau^4 (2D)^{18} \Big(\frac{\mu}{1-\tau}\Big)^6 
\le 2 ,
\end{align*}
finishing the proof.
\end{proof}

We are ready to prove Theorem~\ref{thm:detect-lower}. 

\begin{proof}[Proof of Theorem~\ref{thm:detect-lower}]
To prove that no polynomial of degree at most $D$ strongly separates $\cP$ and $\cQ$, recall the discussion after Theorem~\ref{thm:detect-lower}: 
it suffices to show that the advantage in \eqref{eq:adv-sq-sum} is $O(1)$.
Note that it is bounded by $2$ in Proposition~\ref{prop:adv-bound}, which we now apply.
It suffices to verify that \eqref{eq:detect-lower-cond} implies the assumptions of Proposition~\ref{prop:adv-bound}.
To this end, we consider two cases:
\begin{itemize}
\item
If $\mu = \tilde \Theta(1)$, $n^3 \tau^4 \le n^{-\delta}$, and $D = o \Big( \big(\frac{\log n}{\log \log n}\big)^2 \Big)$, then we can check the first set of conditions in Proposition~\ref{prop:adv-bound}:
$R = \max \big\{ 2D \frac{\mu}{1-\tau}, 1 \big\} = \tilde O(1)$ and 
$2 n^3 \tau^4 (2D)^{13} R^{6 \sqrt{D}} \le n^3 \tau^4 \cdot n^{o(1)} \le 1/2.$

\item
Next, suppose that $\mu \le (\log n)^{-100}$, $n^3 \tau^4 \mu^6 \le n^{-\delta}$, and $D \le (\log n)^{10}$.
We can check the second set of conditions in Proposition~\ref{prop:adv-bound} by further considering two subcases:
\begin{itemize}
\item 
If $\tau \le n^{-1/2}$, then $L = \max\{ n \tau^2 (2D)^4, 1 \} \le (2D)^4$, 
$L \Big(2D \frac{\mu}{1-\tau}\Big)^2 \le \mu^2 (4D)^6 \le 1/2$,
and
$n^3 \tau^4 (2D)^{19} \Big(\frac{\mu}{1-\tau}\Big)^6 \le n^{-\delta} (4D)^{19} \le 1/2$.

\item
If $\tau > n^{-1/2}$, then 
$L = n \tau^2 (2D)^4$ and 
$L \Big(2D \frac{\mu}{1-\tau}\Big)^2 \le n \tau^2 \mu^2 (4D)^6 \le 1/2$ because $n \tau^2 \mu^2 = (n^3 \tau^6 \mu^6)^{1/3} \le (n^3 \tau^4 \mu^6)^{1/3} \le n^{-\delta/3}$.
Finally, $n^3 \tau^4 (2D)^{19} \Big(\frac{\mu}{1-\tau}\Big)^6 \le 1/2$ as in the previous subcase.
\end{itemize}
\end{itemize}
Combining all the cases completes the proof.
\end{proof}

\subsection{Detection upper bound}
\label{sec:detect-upper}

While it suffices to focus on the signed triangle count for the upper bound, we consider cliques with $v \ge 3$ vertices, because some intermediate results hold for a general $v \ge 3$ and may be interesting in their own right.
Define
\begin{equation*}
S_v(A) := \sum_{H \in \binom{[n]}{v}} \prod_{(i,j) \in \binom{H}{2}} \bar{A}_{ij} .
\end{equation*}
Recall the models $\cP$ and $\cQ$ in Definitions~\ref{def:mod-p} and~\ref{def:mod-q} respectively.

\begin{proposition}
\label{prop:q-exp-var}
We have 
$$
\E_{\cQ}[S_v(A)] = 0, \qquad
\Var_{\cQ}(S_v(A)) = \binom{n}{v} . 
$$
\end{proposition}

\begin{proof}
It is clear that $\E_{\cQ}[S_v(A)] = 0$.
Moreover, the variance of $S_v(A)$ under $\cQ$ is equal to  
$$
\E_{\cQ}[S_v(A)^2] = \sum_{H, H' \in \binom{[n]}{v}} \E_{\cQ} \bigg[ \prod_{(i,j) \in \binom{H}{2}} \prod_{(i',j') \in \binom{H'}{2}} \bar A_{ij} \bar A_{i'j'} \bigg] 
= \sum_{H \in \binom{[n]}{v}} \E_{\cQ} \bigg[ \prod_{(i,j) \in \binom{H}{2}} \bar A_{ij}^2 \bigg] 
= \binom{n}{v} .
$$
\end{proof}

For $H \in \binom{[n]}{v}$, define
\begin{equation}
\zeta(z;H) := \left| \left\{ (i,j) \in \binom{H}{2} : \dist(z_i,z_j) \le \tau/2 \right\} \right| .
\label{eq:def-zeta}
\end{equation}

\begin{lemma}
\label{lem:z-exp-lwbd}
Suppose that $3 \le v \le n$ and $0 < \tau \le \frac{1}{2^{v+1} (1+2v^{v-1})}$. 
If $|H| = v$, then 
\begin{equation*}
(-1)^{\binom{v}{2}} \cdot \E_z \Big[ \Big(\frac{\tau-1}{\tau}\Big)^{\zeta(z;H)} \Big] 
\ge \frac{1}{2^{v+1} \tau^{\binom{v}{2}-v+1}} .
\end{equation*}
\end{lemma}

\begin{proof}
Without loss of generality, we assume that $H = [v]$. 
It holds that 
\begin{align}
&(-1)^{\binom{v}{2}} \cdot \E_z \Big[ \Big(\frac{\tau-1}{\tau}\Big)^{\zeta(z;H)} \Big] \notag \\
&= (-1)^{\binom{v}{2}} \cdot \sum_{m=0}^{\binom{v}{2}} \p_z\{ \zeta(z;H) = m \} \cdot \Big(\frac{\tau-1}{\tau}\Big)^m \notag \\
&\ge \p_z \bigg\{ \zeta(z;H) = \binom{v}{2} \bigg\} \cdot \bigg(\frac{1-\tau}{\tau}\bigg)^{\binom{v}{2}} - \sum_{m=0}^{\binom{v}{2} - 1} \p_z\{ \zeta(z;H) = m \} \cdot \bigg(\frac{1-\tau}{\tau}\bigg)^m .
\label{eq:lwbd-1}
\end{align}
We now further bound this quantity from below.

First, conditional on any realization of $z_1$, it holds that
$$
\p_z \big\{ \dist(z_1,z_i) \le \tau/4 \text{ for all } i \in H \mid z_1 \big\} = (\tau/2)^{v-1} .
$$
If $\dist(z_1,z_i) \le \tau/4$ for all $i \in H$, then $\dist(z_i,z_j) \le \tau/2$ for all $i,j \in H$ so that $\zeta(z;H) = \binom{v}{2}$. 
Consequently,
\begin{equation}
\p_z \bigg\{ \zeta(z;H) = \binom{v}{2} \bigg\} 
\cdot \bigg(\frac{1-\tau}{\tau}\bigg)^{\binom{v}{2}}
\ge \Big(\frac{\tau}{2}\Big)^{v-1} \cdot \frac{1 - \tau \binom{v}{2}}{\tau^{\binom{v}{2}}} 
\ge \frac{1}{2^v \tau^{\binom{v}{2}-v+1}} ,
\label{eq:lwbd-2}
\end{equation}
where the last step holds because $\tau \le \frac{1}{v(v-1)}$ by assumption.

Second, we have
\begin{equation}
\sum_{m=0}^{\binom{v}{2} - v} \p_z\{ \zeta(z;H) = m \} \cdot \bigg(\frac{1-\tau}{\tau}\bigg)^m
\le \Bigg( \sum_{m=0}^{\binom{v}{2} - v} \p_z\{ \zeta(z;H) = m \} \Bigg) \cdot \frac{1}{\tau^{\binom{v}{2} - v}}
\le \frac{1}{\tau^{\binom{v}{2} - v}} .
\label{eq:lwbd-3}
\end{equation}

Third, for $\binom{v}{2} - v + 2 \le m \le \binom{v}{2} - 1$ and $\zeta(z;H) = m$, the graph on $H$ with the edge set 
\begin{equation}
\left\{ (i,j) \in \binom{H}{2} : \dist(z_i,z_j) \le \tau/2 \right\}
\label{eq:close-edge-set}
\end{equation}
must be connected. 
As a result, $\dist(z_1,z_i) \le v \tau/2$ for all $i \in H$.
Conditional on any realization of $z_1$, it holds that
$$
\p_z \big\{ \dist(z_1,z_i) \le v \tau/2 \text{ for all } i \in H \mid z_1 \big\} \le (v \tau)^{v-1} .
$$
Then we obtain
\begin{align}
\sum_{m=\binom{v}{2} - v + 2}^{\binom{v}{2} - 1} \p_z\{ \zeta(z;H) = m \} \cdot \bigg(\frac{1-\tau}{\tau}\bigg)^m
&\le \Bigg( \sum_{m=\binom{v}{2} - v + 2}^{\binom{v}{2} - 1} \p_z\{ \zeta(z;H) = m \} \Bigg) \cdot \frac{1}{\tau^{\binom{v}{2} - 1}} \notag \\
&\le (v \tau)^{v-1} \cdot  \frac{1}{\tau^{\binom{v}{2} - 1}} 
= \frac{v^{v-1}}{\tau^{\binom{v}{2} - v}}.
\label{eq:lwbd-4}
\end{align}

Fourth, for $\zeta(z;H) = \binom{v}{2} - v + 1$, the graph on $H$ with the edge set \eqref{eq:close-edge-set} is either connected or has only one isolated vertex $z_{i^*}$. 
Let $j^* = j^*(i^*)$ be any vertex in $H$ not equal to $i^*$. 
Then $\dist(z_{j^*}, z_i) \le v \tau /2$ for all $i \in H \setminus \{i^*\}$. 
Conditional on any realization of $z_{j^*}$, it holds that 
$$
\p_z \big\{ \dist(z_{j^*},z_i) \le v \tau/2 \text{ for all } i \in H \setminus \{i^*\} \mid z_{j^*} \big\} \le (v \tau)^{v-2} .
$$
Then we obtain
\begin{equation}
\p_z \left\{ \zeta(z;H) = \binom{v}{2} - v + 1 \right\} \cdot \bigg(\frac{1-\tau}{\tau}\bigg)^{\binom{v}{2} - v + 1}
\le v (v \tau)^{v-2} \cdot \frac{1}{\tau^{\binom{v}{2} - v + 1}} 
\le \frac{v^{v-1}}{\tau^{\binom{v}{2} - v}} ,
\label{eq:lwbd-5}
\end{equation}
where we used the assumption $v \ge 3$ in the last step.

Finally, combining \eqref{eq:lwbd-1}, \eqref{eq:lwbd-2}, \eqref{eq:lwbd-3}, \eqref{eq:lwbd-4}, and \eqref{eq:lwbd-5} yields that
\begin{equation*}
\bigg| \E_z \Big[ \Big(\frac{\tau-1}{\tau}\Big)^{\zeta(z;H)} \Big] \bigg| 
\ge \frac{1}{2^v \tau^{\binom{v}{2}-v+1}}  - \frac{1}{\tau^{\binom{v}{2} - v}} -  \frac{v^{v-1}}{\tau^{\binom{v}{2} - v}} - \frac{v^{v-1}}{\tau^{\binom{v}{2} - v}} 
\ge \frac{1}{2^{v+1} \tau^{\binom{v}{2}-v+1}} ,
\end{equation*}
since $\frac{1}{2^v \tau} \ge 2(1+2v^{v-1})$ by assumption.
\end{proof}

\begin{proposition}
\label{prop:p-exp}
Suppose that $3 \le v \le n$ and $0 < \tau \le \frac{1}{2^{v+1} (1+2v^{v-1})}$. 
It holds that
$$
\big| \E_{\cP}[S_v(A)] \big| 
\ge \binom{n}{v} \Big(\frac{\mu}{1-\tau}\Big)^{\binom{v}{2}} \frac{\tau^{v-1}}{2^{v+1}} .
$$
\end{proposition}

\begin{proof}
We have 
\begin{align*}
\E_{\cP}[S_v(A)] 
&= \sum_{H \in \binom{[n]}{v}} \E_{\cP} \bigg[ \prod_{(i,j) \in \binom{H}{2}} \bar A_{ij} \bigg] \\
&= \sum_{H \in \binom{[n]}{v}} \E_z \bigg[ \prod_{(i,j) \in \binom{H}{2}} \E_{\cP} [ \bar A_{ij} \mid z] \bigg] \\
&= \frac{1}{(r(1-r))^{\binom{v}{2}/2}} \sum_{H \in \binom{[n]}{v}} \E_z \Big[ (p-r)^{\zeta(z;H)} (q-r)^{\binom{v}{2} - \zeta(z;H)} \Big] .
\end{align*}
Recall that $r = \tau p + (1-\tau) q$ so that $\frac{p-r}{q-r} = \frac{\tau-1}{\tau}$, and $\mu = \frac{p-r}{\sqrt{r(1-r)}}$.
It follows that
$$
\E_{\cP}[S_v(A)] 
= \bigg(\frac{\mu \tau}{\tau-1}\bigg)^{\binom{v}{2}} \sum_{H \in \binom{[n]}{v}} \E_z \Big[ \Big(\frac{\tau-1}{\tau}\Big)^{\zeta(z;H)} \Big] .
$$
Then, by Lemma~\ref{lem:z-exp-lwbd}, we obtain
$$
\big| \E_{\cP}[S_v(A)] \big| 
\ge \bigg(\frac{\mu \tau}{1-\tau}\bigg)^{\binom{v}{2}} \binom{n}{v} \frac{1}{2^{v+1} \tau^{\binom{v}{2}-v+1}} ,
$$
completing the proof.
\end{proof}

\begin{proposition}
\label{prop:p-var}
There is an absolute constant $C>0$ such that
$$
\Var_{\cP}(S_3(A)) \le \frac{C}{r^3 (1-r)^3} \Big( n^4 (\tau p + q + r^2) (r-q)^2 (p-q)^2 + n^3 \big( q^3 + r^6 + \tau p q^2 + \tau p r^4 + \tau^2 p^3 \big) \Big) .
$$
\end{proposition}

\begin{proof}
For brevity, let $\sigma := \sqrt{r(1-r)}$ in this proof.
The variance of $S_3(A)$ under $\cP$ is 
\begin{align*}
&\Var_{\cP}(S_3(A)) = \E_{\cP}[S_3(A)^2] - \E_{\cP}[S_3(A)]^2 \\
&= \sum_{H, H' \in \binom{[n]}{3}} \E_{\cP} \bigg[ \prod_{(i,j) \in \binom{H}{2}} \prod_{(i',j') \in \binom{H'}{2}} \bar A_{ij} \bar A_{i'j'} \bigg] - \sum_{H, H' \in \binom{[n]}{3}} \E_{\cP} \bigg[ \prod_{(i,j) \in \binom{H}{2}} \bar A_{ij} \bigg] \, \E_{\cP} \bigg[ \prod_{(i',j') \in \binom{H'}{2}} \bar A_{i'j'} \bigg] .
\end{align*}
For fixed $H, H' \subseteq[n]$ with $|H| = |H'| = 3$, consider the following cases:
\begin{itemize}
\item
$|H \cap H'| = 0$: 
We have that $\{z_i : i \in H\}$ and $\{z_{i'} : i' \in H'\}$ are independent, and consequently, $\{\bar A_{ij} : (i,j) \in \binom{H}{3}\}$ and $\{\bar A_{i'j'} : (i',j') \in \binom{H'}{3}\}$ are independent. 
Therefore, 
$$
\E_{\cP} \bigg[ \prod_{(i,j) \in \binom{H}{2}} \prod_{(i',j') \in \binom{H'}{2}} \bar A_{ij} \bar A_{i'j'} \bigg] - \E_{\cP} \bigg[ \prod_{(i,j) \in \binom{H}{2}} \bar A_{ij} \bigg] \, \E_{\cP} \bigg[ \prod_{(i',j') \in \binom{H'}{2}} \bar A_{i'j'} \bigg] 
= 0 .
$$

\item
$|H \cap H'| = 1$: 
Suppose $H \cap H' = \{i^*\}$. 
Conditional on $z_{i^*}$, we have that $\{z_i : i \in H \setminus \{i^*\}\}$ and $\{z_{i'} : i' \in H' \setminus \{i^*\}\}$ are independent. 
Moreover, $\{\bar A_{ij} : (i,j) \in \binom{H}{3}\}$ and $\{\bar A_{i'j'} : (i',j') \in \binom{H'}{3}\}$ are conditionally independent, and their distributions are not changed by the conditioning on $z_{i^*}$. 
Therefore, we still have
$$
\E_{\cP} \bigg[ \prod_{(i,j) \in \binom{H}{2}} \prod_{(i',j') \in \binom{H'}{2}} \bar A_{ij} \bar A_{i'j'} \bigg] - \E_{\cP} \bigg[ \prod_{(i,j) \in \binom{H}{2}} \bar A_{ij} \bigg] \, \E_{\cP} \bigg[ \prod_{(i',j') \in \binom{H'}{2}} \bar A_{i'j'} \bigg] 
= 0 .
$$

\item
$|H \cap H'| = 2$: 
Without loss of generality, suppose that $H = \{1, 2, 3\}$ and $H' = \{1, 2, 4\}$. 
Then
\begin{align*}
&\E_{\cP} \bigg[ \prod_{(i,j) \in \binom{H}{2}} \prod_{(i',j') \in \binom{H'}{2}} \bar A_{ij} \bar A_{i'j'} \bigg] \\
&= \E_z \Big[ \E_{\cP} [\bar A_{12}^2 \mid z] \cdot \E_{\cP} [\bar A_{13} \mid z] \cdot \E_{\cP} [\bar A_{23} \mid z] \cdot \E_{\cP} [\bar A_{14} \mid z] \cdot \E_{\cP} [\bar A_{24} \mid z] \Big] \\
&= \frac{1}{\sigma^6} \, \E_z \Big[ \big( p (1-r)^2 + (1-p) r^2 \big)^{\bbone\{\dist(z_1,z_2) \le \tau/2\}} \big( q (1-r)^2 + (1-q) r^2 \big)^{\bbone\{\dist(z_1,z_2) > \tau/2\}} \\
& \qquad \qquad \quad \cdot \big( p - r \big)^{\tilde \zeta(z)} \big( q - r \big)^{4 - \tilde \zeta(z)} \Big] ,
\end{align*}
where $\tilde \zeta(z) := |\{(i,j) \in \{(1,3), (2,3), (1,4), (2,4) : \dist(z_i,z_j) \le \tau/2\}|$. 
We have the following:
\begin{itemize}
\item
It is obvious that 
$$
\p_z\big\{ \dist(z_1,z_2) \le \tau/2, \, \tilde \zeta(z) = 0 \big\} \le \tau, \qquad
\p_z\big\{ \dist(z_1,z_2) > \tau/2, \, \tilde \zeta(z) = 0 \big\} \le 1 .
$$

\item
Condition on any realization of $(z_1,z_2)$.
If $1 \le \tilde \zeta(z) \le 2$, then one of the following four events must occur: (1) $\dist(z_1,z_3) \le \tau/2$, (2) $\dist(z_2,z_3) \le \tau/2$, (3) $\dist(z_1,z_4) \le \tau/2$, or (4) $\dist(z_2,z_4) \le \tau/2$; this holds with conditional probability at most $4 \tau$. 
Therefore,
$$
\p_z\big\{ 1 \le \tilde \zeta(z) \le 2 \mid z_1, z_2 \big\} \le 4 \tau ,
$$
so we obtain 
$$
\p_z\big\{ \dist(z_1,z_2) \le \tau/2, \, 1 \le \tilde \zeta(z) \le 2 \big\} \le 4 \tau^2 , \qquad 
\p_z\big\{ \dist(z_1,z_2) > \tau/2, \, 1 \le \tilde \zeta(z) \le 2 \big\} \le 4 \tau.
$$

\item
Condition on any realization of $(z_1,z_2)$.
If $3 \le \tilde \zeta(z) \le 4$, then one of the following four events must occur: (1) $\dist(z_1,z_3) \le \tau/2$ and $\dist(z_1,z_4) \le \tau/2$, (2) $\dist(z_1,z_3) \le \tau/2$ and $\dist(z_2,z_4) \le \tau/2$, (3) $\dist(z_2,z_3) \le \tau/2$ and $\dist(z_1,z_4) \le \tau/2$, (4) $\dist(z_2,z_3) \le \tau/2$ and $\dist(z_2,z_4) \le \tau/2$; this holds with conditional probability at most $4 \tau^2$. 
Therefore, 
$$
\p_z\big\{ 3 \le \tilde \zeta(z) \le 4 \mid z_1, z_2 \big\} \le 4 \tau^2 ,
$$
so we obtain
$$
\p_z\big\{ \dist(z_1,z_2) \le \tau/2, \, 3 \le \tilde \zeta(z) \le 4 \big\} \le 4 \tau^3 , \qquad
\p_z\big\{ \dist(z_1,z_2) > \tau/2, \, 3 \le \tilde \zeta(z) \le 4 \big\} \le 4 \tau^2.
$$
\end{itemize}
Combining the above bounds, we see that
\begin{align*}
&\E_{\cP} \bigg[ \prod_{(i,j) \in \binom{H}{2}} \prod_{(i',j') \in \binom{H'}{2}} \bar A_{ij} \bar A_{i'j'} \bigg] \\
&= \frac{1}{\sigma^6}  \bigg[ \sum_{\ell=0}^4 \p_z \big\{ \dist(z_1,z_2) \le \tau/2, \, \tilde \zeta(z) = \ell \big\} \cdot \big( p (1-r)^2 + (1-p) r^2 \big) \big( p - r \big)^\ell \big( q - r \big)^{4 - \ell} \\
& \qquad \quad + \sum_{\ell=0}^4 \p_z \big\{ \dist(z_1,z_2) > \tau/2, \, \tilde \zeta(z) = \ell \big\} \cdot \big( q (1-r)^2 + (1-q) r^2 \big) \big( p - r \big)^\ell \big( q - r \big)^{4 - \ell} \bigg] \\
&\le \frac{1}{\sigma^6}  \bigg[ \big( p (1-r)^2 + (1-p) r^2 \big) \Big( \tau \big( r - q \big)^4
+ 4 \tau^2 \big( p - r \big)^2 \big( r - q \big)^2
+ 4 \tau^3 \big( p - r \big)^4 \Big) \\
& \qquad \quad + \big( q (1-r)^2 + (1-q) r^2 \big) \Big( \big( r - q \big)^4
+ 4 \tau \big( p - r \big)^2 \big( r - q \big)^2
+ 4 \tau^2 \big( p - r \big)^4 \Big) \bigg] ,
\end{align*}
where we omitted negative terms where $\ell$ is odd.
Recall that $\tau(p-r) = (1-\tau)(r-q) \le r-q$. 
Also, the condition $0 < q < r < p < 1$ implies that
\begin{equation}
p (1-r)^2 + (1-p) r^2 \le p + r^2 \le 2p , \qquad
q (1-r)^2 + (1-q) r^2 \le q + r^2 .
\label{eq:pqr-cond}
\end{equation}
It then follows that
\begin{align*}
&\E_{\cP} \bigg[ \prod_{(i,j) \in \binom{H}{2}} \prod_{(i',j') \in \binom{H'}{2}} \bar A_{ij} \bar A_{i'j'} \bigg] \\
&\le \frac{1}{\sigma^6}  \Big( \tau \big( p (1-r)^2 + (1-p) r^2 \big) + \big( q (1-r)^2 + (1-q) r^2 \big) \Big) 
\Big( \big( r - q \big)^2
+ 2 \tau \big( p - r \big)^2 \Big)^2 \\ 
&\le \frac{1}{\sigma^6}  ( 2 \tau p + q + r^2 ) 
\Big( ( r - q )^2
+ 2 (r-q) ( p - r ) \Big)^2 \\
&\le \frac{4}{\sigma^6}  ( 2 \tau p + q + r^2 ) (r-q)^2 ( p-q )^2 .
\end{align*}

\item
$|H \cap H'| = 3$: 
Without loss of generality, suppose that $H = H' = \{1,2,3\}$. Then
\begin{align*}
&\E_{\cP} \bigg[ \prod_{(i,j) \in \binom{H}{2}} \prod_{(i',j') \in \binom{H'}{2}} \bar A_{ij} \bar A_{i'j'} \bigg] 
= \E_{\cP} \bigg[ \prod_{(i,j) \in \binom{H}{2}} \bar A_{ij}^2 \bigg] 
= \E_z \Bigg[ \prod_{(i,j) \in \binom{H}{2}} \E_{\cP} [\bar A_{ij}^2 \mid z] \Bigg] \\
&= \frac{1}{\sigma^6} \, \E_z \Big[ \big( p (1-r)^2 + (1-p) r^2 \big)^{\zeta(z; H)} \big( q (1-r)^2 + (1-q) r^2 \big)^{3 - \zeta(z; H)} \Big] .
\end{align*}
We have $\p_z\{\dist(z_1,z_2) \le \tau/2\} = \tau$, so by symmetry, $\p_z\{\zeta(z;H) = 1\} \le 3 \tau$.
Moreover, let us condition on any realization of $z_1$. 
If $\zeta(z;H) \ge 2$, then $\dist(z_1,z_2) \le \tau$ and $\dist(z_1,z_3) \le \tau$, which occurs with conditional probability at most $(2\tau)^2$. 
Therefore, $\p_z\{\zeta(z;H) \ge 2\} \le (2\tau)^2$. 
We then obtain
\begin{align*}
&\E_{\cP} \bigg[ \prod_{(i,j) \in \binom{H}{2}} \prod_{(i',j') \in \binom{H'}{2}} \bar A_{ij} \bar A_{i'j'} \bigg] \\
&= \frac{1}{\sigma^6}  \sum_{\ell=0}^3 \p_z\{\zeta(z;H) = \ell\} \cdot \big( p (1-r)^2 + (1-p) r^2 \big)^{\ell} \big( q (1-r)^2 + (1-q) r^2 \big)^{3 - \ell} \\
&\le \frac{1}{\sigma^6}  \Big[ \big( q (1-r)^2 + (1-q) r^2 \big)^3 
+ 3 \tau \big( p (1-r)^2 + (1-p) r^2 \big) \big( q (1-r)^2 + (1-q) r^2 \big)^2 \\
&\qquad + (2\tau)^2 \Big( \big( p (1-r)^2 + (1-p) r^2 \big)^2 \big( q (1-r)^2 + (1-q) r^2 \big) + \big( p (1-r)^2 + (1-p) r^2 \big)^3 \Big) \Big] \\
&\le \frac{1}{\sigma^6}  \Big[ ( q + r^2 )^3 
+ 6 \tau p ( q + r^2 )^2 
+ 64 \tau^2 p^3 \Big] ,
\end{align*}
where the last step follows from \eqref{eq:pqr-cond}.
\end{itemize}
In summary, we have 
\begin{align*}
&\Var_{\cP}(S_3(A)) \\
&\le \sum_{\substack{H, H' \in \binom{[n]}{3} : \\ |H \cap H'| = 2}} \frac{4}{\sigma^6}  ( 2 \tau p + q + r^2 ) (r-q)^2 ( p-q )^2 
+ \sum_{H \in \binom{[n]}{3}} \frac{1}{\sigma^6}  \Big[ ( q + r^2 )^3 
+ 6 \tau p ( q + r^2 )^2 
+ 64 \tau^2 p^3 \Big] \\
&\le \frac{C}{\sigma^6} \Big( n^4 (\tau p + q + r^2) (r-q)^2 (p-q)^2 + n^3 \big( q^3 + r^6 + \tau p q^2 + \tau p r^4 + \tau^2 p^3 \big) \Big) 
\end{align*}
for an absolute constant $C > 0$.
\end{proof}

We are ready to prove Theorem~\ref{thm:detect-upper}.

\begin{proof}[Proof of Theorem~\ref{thm:detect-upper}]
In view of the assumption $p \ge C q$ for a constant $C > 1$ and the definition $r = \tau p + (1-\tau)q$ where $\tau \le 1/2$, we have $\mu = \frac{p-r}{\sqrt{r(1-r)}} = \Theta(p/r^{1/2})$.
By Propositions~\ref{prop:q-exp-var} and \ref{prop:p-exp}, we obtain
$$
\left| \E_{\cP}[S_3(A)] - \E_{\cQ}[S_3(A)] \right| \ge \binom{n}{3} \Big(\frac{\mu}{1-\tau}\Big)^3 \frac{\tau^2}{16}
= \Omega( n^3 \tau^2 p^3/r^{3/2} ) 
$$
and
$$
\Var_{\cQ}(S_3(A)) \le n^3 .
$$
Moreover, the definition $r = \tau p + (1-\tau)q$ where $\tau \le 1/2$ implies $r - q = \frac{\tau}{1-\tau} (p-r) \le 2 \tau p$.
Hence, the bound in Proposition~\ref{prop:p-var} simplifies to
$$
\Var_{\cP}(S_3(A)) 
= O \Big( \frac{1}{r^3} (n^4 \tau^3 p^5 + n^4 \tau^2 p^4 r + n^3 r^3 + n^3 \tau p r^2 + n^3 \tau^2 p^3) \Big) .
$$
Consequently, for $S_3(A)$ to strongly separate $\cP$ and $\cQ$, it suffices to have
$$
\sqrt{n^3 + n^4 \tau^3 p^5 / r^3 + n^4 \tau^2 p^4 / r^2 + n^3 \tau p / r + n^3 \tau^2 p^3 / r^3} = o( n^3 \tau^2 p^3 / r^{3/2} ) ,
$$
which (by dividing the square of the RHS by each term on the LHS) is equivalent to 
$$
\min \left\{ n^3 \tau^4 p^6/r^3 , n^2 \tau p , n^2 \tau^2 p^2/r , n^3 \tau^3 p^5/r^2 , n^3 \tau^2 p^3 \right\} = \omega(1) .
$$
The first and the last quantity on the LHS are assumed to be $\omega(1)$ in \eqref{eq:detect-upper-cond}, and the middle three quantities are $\omega(1)$ because $(n^2 \tau p)^3 \ge (n^3 \tau^2 p^3)^2$, $(n^2 \tau^2 p^2/r)^3 \ge (n^3 \tau^4 p^6/r^3) \cdot (n^3 \tau^2 p^3)$, and $(n^3 \tau^3 p^5/r^2)^3 \ge (n^3 \tau^4 p^6/r^3)^2 \cdot (n^3 \tau^2 p^3)$.
\end{proof}

\subsection{Recovery lower bound}
\label{sec:recover-lower}

We first prove Proposition~\ref{prop:lower-bound-estimation-general}.

\begin{proof}[Proof of Proposition~\ref{prop:lower-bound-estimation-general}]
Recall Definition~\ref{def:general-binary-model}.
Define $Z, Y \in \R^N$ by $Z_i := \frac{B_i - q}{\sqrt{q(1-q)}}$ and $Y_i := \frac{A_i - q}{\sqrt{q(1-q)}}$ for $i \in [N]$.
Since any polynomial of in $(A_i)_{i \in [N]}$ is also a polynomial in $(Y_i)_{i \in [N]}$ of the same degree (and vice versa), we have
$$
\Corr_{\le D} 
= \sup_{\substack{f \in \R[Y]_{\le D}, \\ \E[f(Y)^2] \ne 0}} \frac{\E[f(Y) \cdot \chi]}{\sqrt{\E[f(Y)^2]}} . 
$$
For $f \in \R[Y]_{\le D}$, we can write 
$$
f(Y) = \sum_{\alpha \subseteq [N] \,:\, |\alpha| \le D} \hat f_\alpha \, Y^\alpha ,
$$
where $Y^\alpha := \prod_{i \in \alpha} Y_i$ and $\hat f_\alpha$ denotes the coefficient of $f$ in the basis $\{Y^\alpha \,:\, \alpha \subseteq [N]\}$.

Recall that $\lambda = \frac{p-q}{\sqrt{q(1-q)}}$. 
Then we have $\E[Y_i \mid i \in W] = \lambda$.
It holds that
$$
\E[f(Y) \, \chi] = \sum_{\alpha \subseteq [N] \,:\, |\alpha| \le D} \hat f_\alpha \, \E[Y^\alpha \, \chi]
=: \langle \hat f, v \rangle ,
$$
where 
\begin{equation}
v_\alpha := \E[Y^\alpha \, \chi]
= \p\{\alpha \cup \{1\} \subseteq W\} \cdot \lambda^{|\alpha|}. 
\label{eq:def-v-alpha}
\end{equation}

Moreover, by Jensen's inequality, 
$$
\E[f(Y)^2] \ge \E \Big[ \big( \E[ f(Y) \mid Z ] \big)^2 \Big] 
=: \E[ g(Z)^2 ] ,
$$
where 
$$
g(Z) := \E[ f(Y) \mid Z ]
= \sum_{\alpha \subseteq [N] \,:\, |\alpha| \le D} \hat f_\alpha \, \E[ Y^\alpha \mid Z ] .
$$
Moreover, we have 
$$
\E[ Y^\alpha \mid Z ] 
= \sum_{\beta \subseteq \alpha} \p\{ \alpha \setminus W = \beta \} \cdot Z^\beta \lambda^{|\alpha|-|\beta|} .
$$
Together with the definitions of $g(Z)$ and $P_{\alpha \beta}$, this implies that
\begin{align*}
g(Z) = 
\sum_{\beta \subseteq [N] \,:\, |\beta| \le D} Z^\beta \sum_{\alpha \subseteq [N] \,:\, \alpha \supseteq \beta, \, |\alpha| \le D} \hat f_\alpha \, \lambda^{|\alpha|-|\beta|} \, P_{\alpha \beta} 
= \sum_{\beta \subseteq [N] \,:\, |\beta| \le D} \hat g_\beta Z^\beta ,
\end{align*}
where
$$
\hat g_\beta = \sum_{\alpha \subseteq [N] \,:\, \alpha \supseteq \beta, \, |\alpha| \le D} \hat f_\alpha \, \lambda^{|\alpha|-|\beta|} \, P_{\alpha \beta} .
$$
Therefore, we have $\hat g = M \hat f$ where the matrix $M$ is indexed by $\beta, \alpha \subseteq [N]$ with $|\beta|, |\alpha| \le D$, and $M$ is defined by 
\begin{equation}
M_{\beta \alpha} := \bbone_{\beta \subseteq \alpha} \, \lambda^{|\alpha|-|\beta|} \, P_{\alpha \beta} .
\label{eq:def-M-beta-alpha}
\end{equation}

Since the basis $\{Z^\beta : \beta \subseteq [N]\}$ is orthonormal, we have $\E[f(Y)^2] \ge \E[ g(Z)^2 ] = \|\hat g\|^2$. 
Note that $M$ is invertible because it is upper triangular with nonzero diagonal entries. This means $\E[f(Y) \, \chi] = \langle \hat f , v \rangle = v^\top M^{-1} \hat g$. 
Therefore, 
\begin{align}
\Corr_{\le D} 
= \sup_{f \in \R[Y]_{\le D}, \, \E[f(Y)^2] \ne 0} \frac{ \E[f(Y) \, \chi] }{ \sqrt{\E[f(Y)^2]} } 
\le \sup_{\hat g \ne 0} \frac{v^\top M^{-1} \hat g}{\|\hat g\|} 
= \|v^\top M^{-1}\| =: \|w\| ,
\label{eq:corr-bd-w}
\end{align}
where $w$ is defined by $w^\top M = v^\top$. 
Moreover, we can solve for $w$ recursively as
\begin{equation}
w_\alpha = \frac{1}{M_{\alpha \alpha}} \Big( v_\alpha - \sum_{\beta \subsetneq \alpha} w_\beta M_{\beta \alpha} \Big) .
\label{eq:def-w-alpha}
\end{equation}

Let us define $\rho_\alpha := w_\alpha \lambda^{-|\alpha|}$. 
Then by \eqref{eq:def-v-alpha}, \eqref{eq:def-M-beta-alpha}, and \eqref{eq:def-w-alpha},
$$
\rho_\varnothing = w_\varnothing = v_\varnothing = \p\{ 1 \in W \} ,
$$
and
\begin{align*}
\rho_\alpha 
= \frac{1}{M_{\alpha \alpha}} \bigg( \lambda^{-|\alpha|} v_\alpha - \sum_{\beta \subsetneq \alpha} \lambda^{-(|\alpha| - |\beta|)} \rho_\beta M_{\beta \alpha} \bigg) 
= \frac{1}{P_{\alpha \alpha}} \bigg( \p\{\alpha \cup \{1\} \subseteq W\} - \sum_{\beta \subsetneq \alpha} \rho_\beta \, P_{\alpha \beta} \bigg) .
\end{align*}
The conclusion then follows from \eqref{eq:corr-bd-w} together with the definition of $\rho_\alpha$.
\end{proof}

We assume model $\cP$ in the rest of this section.

\begin{lemma}
\label{lem:empty-intersection}
For any $\alpha \subseteq \binom{[n]}{2}$, we have 
$$
P_{\alpha \alpha} \ge 1 - \tau \, |\alpha| \, (2 |\alpha| - 1) 
\ge 1 - 2 \tau |\alpha|^2 .
$$
\end{lemma}

\begin{proof}
Recall that $P_{\alpha \alpha} = \p\{\alpha \setminus W = \alpha\} = \p\{\alpha \cap W = \varnothing\}$.
By \eqref{eq:W-meaning}, we have
\begin{align*}
\p\{\alpha \cap W = \varnothing\} 
&= \p\{\dist(z_i, z_j) > \tau/2 \text{ for all } (i,j) \in \alpha\} \\
&\ge \p\{\dist(z_i, z_j) > \tau/2 \text{ for all } i,j \in V(\alpha), \, i \ne j\} \\
&\ge \prod_{m=1}^{|V(\alpha)|-1} (1 - m \tau) 
\ge 1 - \tau \, \binom{|V(\alpha)|}{2} .
\end{align*}
Since $|V(\alpha)| \le 2 \, |\alpha|$, the desired bound follows.
\end{proof}

\begin{lemma}
\label{lem:zero-cumulant}
For $\alpha \ne \varnothing$, suppose that $P_{\beta \beta} > 0$ for all $\beta \subseteq\alpha$.
We have $\rho_\alpha = 0$ in either of the following situations:
\begin{itemize}
\item
$1 \notin V(\alpha)$ or $2 \notin V(\alpha)$;
\item
$\alpha$ is a disconnected graph.
\end{itemize}
\end{lemma}

\begin{proof}
To facilitate the proof, we consider two cases that are split in a different way from the two cases in the statement of the lemma:
\begin{enumerate}
\item
either $1 \notin V(\alpha)$, $2 \notin V(\alpha)$, or vertices $1$ and $2$ are in different connected components of $\alpha$;

\item
$\alpha$ is disconnected, but vertices $1$ and $2$ are in the same connected component of $\alpha$.
\end{enumerate}
Note that $|\alpha| \ge 1$ in Case~1 and $|\alpha| \ge 2$ in Case~2.
For both cases, we prove  $\rho_\alpha = 0$ by induction on $|\alpha|$.
Each proof will establish the base case and the induction step simultaneously.

\paragraph{Case 1:}
Let $G$ be the union of the graph $\alpha$ and the (potentially isolated) vertices $1$ and $2$.
Let $G_1$ denote the connected component of $G$ that contains $1$, and let $G_2$ be the complement of $G_1$ in $G$.
Let $E(G_i)$ denote the (potentially empty) edge set of $G_i$ for $i = 1,2$.
By \eqref{eq:W-meaning} and the fact that $G_1$ is not connected to $G_2$ , we see that
\begin{align*}
\p\{\alpha \cup \{(1, 2)\} \subseteq W\}
&= \p\{E(G_1) \subseteq W\} \cdot \p\{(1, 2) \in W\} \cdot \p\{E(G_2) \subseteq W\} \\
&= \tau \, \p\{E(G_1) \subseteq W\} \cdot \p\{E(G_2) \subseteq W\} \\
&= \tau \, \p\{\alpha \subseteq W\} 
= \tau \, P_{\alpha \varnothing} .
\end{align*}
In the base case $|\alpha|=1$, there is no nonempty $\beta \subsetneq \alpha$; in the case $|\alpha|>1$, if $\varnothing \subsetneq \beta \subsetneq \alpha$, then $\rho_\beta = 0$ by the induction hypothesis. 
Combining these facts with \eqref{eq:rho-recursion} gives
\begin{align*}
\rho_\alpha \, P_{\alpha \alpha}
= \p\{\alpha \cup \{(1, 2)\} \subseteq W\} - \sum_{\beta \subsetneq \alpha} \rho_\beta \, P_{\alpha \beta} 
= \tau \, P_{\alpha \varnothing} - \rho_\varnothing \, P_{\alpha \varnothing} = 0. 
\end{align*}
Since $P_{\alpha \alpha} > 0$ by assumption, we conclude that $\rho_\alpha = 0$.

\paragraph{Case 2:}
Consider a subgraph $\beta \subsetneq \alpha$.
If $1 \notin V(\alpha)$, $2 \notin V(\alpha)$, or vertices $1$ and $2$ are in different connected components of $\beta$, then $\rho_\beta = 0$ by Case~1 above.
If $\beta$ is disconnected while $1$ and $2$ are in the same connected component of $\beta$, then $\rho_\beta = 0$ by the induction hypothesis (and there is simply no such $\beta$ in the base case $|\alpha| = 2$).
Therefore, if $\rho_\beta \ne 0$, then $\beta$ must be a connected graph containing both vertices $1$ and $2$.

Let $\gamma$ be the connected component of $\alpha$ that contains vertices $1$ and $2$.
We obtain from \eqref{eq:rho-recursion} that 
\begin{align*}
\rho_\alpha \, P_{\alpha \alpha}
&= \p\{\alpha \cup \{(1, 2)\} \subseteq W\} - \sum_{\beta \subseteq \gamma} \rho_\beta \, P_{\alpha \beta} \\
&= \p\{\alpha \setminus \gamma \subseteq W\} \, \cdot \p\{\gamma \cup \{(1, 2)\} \subseteq W\} 
- \rho_{\gamma} \, P_{\alpha \gamma}
- \sum_{\beta \subsetneq \gamma} \rho_\beta \, P_{\alpha \beta} .
\end{align*}
Furthermore, using \eqref{eq:rho-recursion} again yields
\begin{align*}
\rho_\gamma 
&= \frac{ \p\{\gamma \cup \{(1, 2)\} \subseteq W\} }{P_{\gamma \gamma}} - \sum_{\beta \subsetneq \gamma} \rho_\beta \, \frac{P_{\gamma \beta}}{P_{\gamma \gamma}} .
\end{align*}
The above two equations together imply
\begin{align}
\rho_\alpha \, P_{\alpha \alpha}
= \p\{\gamma \cup \{(1, 2)\} \subseteq W\} \, \bigg( \p\{\alpha \setminus \gamma \subseteq W\} - \frac{P_{\alpha \gamma}}{P_{\gamma \gamma}} \bigg) 
- \sum_{\beta \subsetneq \gamma} \rho_\beta \, \bigg( P_{\alpha \beta} - P_{\alpha \gamma} \frac{P_{\gamma \beta}}{P_{\gamma \gamma}} \bigg) .
\label{eq:rho-alpha-P-alpha}
\end{align}
In view of the assumption $P_{\alpha \alpha}>0$, it remains to show that the two terms in the brackets are zero. 

First, by definition,
$$
P_{\alpha \gamma}
= \p\{\alpha \setminus W = \gamma\}
= \p\{\alpha \setminus \gamma \subseteq W, \, \gamma \cap W = \varnothing\} .
$$
Since $\gamma$ is disconnected from $\alpha \setminus \gamma$ by construction, we obtain
\begin{align*}
P_{\alpha \gamma}
= \p\{\alpha \setminus \gamma \subseteq W\} \cdot \p\{\gamma \cap W = \varnothing\} 
= \p\{\alpha \setminus \gamma \subseteq W\} \cdot P_{\gamma \gamma} .
\end{align*}
Hence the first term in \eqref{eq:rho-alpha-P-alpha} is zero. 

Second, since $\beta \subseteq \gamma \subseteq \alpha$, we have $\alpha \setminus W = \beta$ if and only if $\gamma \setminus W = \beta$ and $\alpha \setminus \gamma \subseteq W$. 
Thus
\begin{align*}
\frac{P_{\alpha \beta}}{P_{\gamma \beta}}
= \frac{\p\{\gamma \setminus W = \beta, \, \alpha \setminus \gamma \subseteq W\}}{\p\{\gamma \setminus W = \beta\}} 
= \p\{\alpha \setminus \gamma \subseteq W\} ,
\end{align*}
where the last equality holds because $\gamma$ is disconnected from $\alpha \setminus \gamma$ and thus $\alpha \setminus \gamma \subseteq W$ is independent of $\gamma \setminus W = \beta$. 
Note that this ratio does not depend on $\beta$, so we can set $\beta = \gamma$ and obtain
$$
\frac{P_{\alpha \gamma}}{P_{\gamma \gamma}} = \p\{\alpha \setminus \gamma \subseteq W\} = \frac{P_{\alpha \beta}}{P_{\gamma \beta}} .
$$
Consequently, $P_{\alpha \beta} - P_{\alpha \gamma} \frac{P_{\gamma \beta}}{P_{\gamma \gamma}} = 0$ and so the second term in \eqref{eq:rho-alpha-P-alpha} is also zero. 
\end{proof}

\begin{lemma}
Fix $\alpha \subseteq \binom{[n]}{2}$ such that $1, 2 \in V(\alpha)$ and $\alpha$ is a connected graph.
We have
\begin{align}
\p\{\alpha \cup \{(1, 2)\} \subseteq W\}
\le P_{\alpha \varnothing} 
\le (\tau \, |V(\alpha)|)^{|V(\alpha)|-1} 
\label{eq:x-exp-bd-1} .
\end{align}
Moreover, fix $\beta \subseteq \binom{[n]}{2}$ such that $\varnothing \subsetneq \beta \subsetneq \alpha$ and $1, 2 \in V(\beta)$. 
We have 
\begin{equation}
P_{\alpha \beta} \le (\tau \, |V(\alpha)|)^{|V(\alpha)| - |V(\beta)|} .
\label{eq:x-exp-bd-2}
\end{equation}
\end{lemma}

\begin{proof}
By the definition of $W$, it holds that
\begin{align*}
&\p\{\alpha \cup \{(1, 2)\} \subseteq W\} 
= \p \big\{ \dist(z_i,z_j) \le \tau/2 \text{ for all } (i,j) \in \alpha \cup \{(1,2)\} \big\} , \\
&P_{\alpha \varnothing}
= \p\{\alpha \setminus W = \varnothing\} 
= \p \big\{ \dist(z_i,z_j) \le \tau/2 \text{ for all } (i,j) \in \alpha \big\} ,
\end{align*}
so the first inequality in \eqref{eq:x-exp-bd-1} is obvious.
Next, suppose $\dist(z_i,z_j) \le \tau/2$ for all $(i,j) \in \alpha$. 
Fix $\ell \in V(\alpha)$. 
Since $\alpha$ is a connected graph, there is a path from $\ell$ to any $i \in V(\alpha)$ that has length at most $|V(\alpha)|$. As a result, $\dist(z_i,z_\ell) \le |V(\alpha)| \cdot \tau/2$ for all $i \in V(\alpha)$. 
Conditional on any realization of $z_\ell$, the probability that $\dist(z_i,z_\ell) \le |V(\alpha)| \cdot \tau/2$ for all $i \in V(\alpha)$ is at most $(\tau \, |V(\alpha)|)^{|V(\alpha)|-1}$. 
Hence \eqref{eq:x-exp-bd-1} follows.

Next, we have
\begin{align*}
P_{\alpha \beta}
= \p\{\alpha \setminus W = \beta\} 
\le \p \big\{ \dist(z_i,z_j) \le \tau/2 \text{ for all } (i,j) \in \alpha \setminus \beta \big\} .
\end{align*}
Suppose $\dist(z_i,z_j) \le \tau/2$ for all $(i,j) \in \alpha \setminus \beta$. 
Fix any vertex $s \in V(\alpha) \setminus V(\beta)$. 
We claim that there is a path from $s$ to a vertex $t_s \in V(\beta)$ which has length at most $|V(\alpha)|$ and lies entirely in $\alpha \setminus \beta$. 
Given the claim, it follows that $\dist(z_s,z_{t_s}) \le |V(\alpha)| \cdot \tau/2$. 
Now, conditional on any realization of $\{z_i : i \in V(\beta)\}$, the probability that $\dist(z_s,z_{t_s}) \le |V(\alpha)| \cdot \tau/2$ for all $s \in V(\alpha) \setminus V(\beta)$ is at most $(\tau \, |V(\alpha)|)^{|V(\alpha)| - |V(\beta)|}$. 
Hence \eqref{eq:x-exp-bd-2} follows. 

It remains to prove the claim. Pick any $r \in V(\beta)$ and take a path from $s$ to $r$ in the graph $\alpha \cup \{(1,2)\}$. 
Since $1,2 \in V(\beta)$, the first edge in the path (that is, the edge adjacent to $s \in V(\alpha) \setminus V(\beta)$) is neither $(1,2)$ nor belongs to $\beta$, and so it must belong to $\alpha \setminus \beta$. 
Following the path, we can find the first vertex $t_s$ that is in $V(\beta)$. 
For the same reason, all edges between $s$ and $t_s$ must belong to $\alpha \setminus \beta$, proving the claim.
\end{proof}

\begin{lemma}
\label{lem:kappa-bd}
Fix $\alpha \subseteq \binom{[n]}{2}$ such that $1, 2 \in V(\alpha)$ and $\alpha$ is a connected graph. 
Suppose that $\tau |\alpha|^4 \le 0.1$. 
If $\alpha^*$ consists of the single edge $(1, 2)$, then $|\rho_{\alpha^*}| \le \tau$. 
More generally, 
\begin{equation*}
|\rho_\alpha| 
\le (1 + \tau |\alpha|^4) \, (|\alpha| + 1)^{|\alpha|} \, (\tau \, |V(\alpha)|)^{|V(\alpha)|-1} .
\end{equation*}
\end{lemma}

\begin{proof}
We prove the result by induction on $|\alpha|$. 
First, consider the base case $|\alpha| = 1$. We must have $(1,2) \in \alpha$ since $1, 2 \in V(\alpha)$.
By \eqref{eq:rho-recursion} and Lemma~\ref{lem:empty-intersection}, we have
$$
|\rho_{\alpha^*}|
\le \frac{1}{1-\tau} \bigg( \p\{\alpha^* \subseteq W\} - \rho_\varnothing \, P_{\alpha^* \varnothing} \bigg)
= \frac{\tau - \tau^2}{1-\tau} = \tau ,
$$
since $\p\{\alpha^* \subseteq W\} = P_{\alpha^* \varnothing} = \p\{ \dist(1, 2) \le \tau/2 \} = \tau$.

Next, fix $\alpha \subseteq \binom{[n]}{2}$ with $|\alpha| \ge 2$. 
Assume $|\rho_\beta| \le (1 + \tau |\beta|^4) \, (|\beta| + 1)^{|\beta|} \, (\tau \, |V(\beta)|)^{|V(\beta)|-1}$ for all $\beta \subsetneq \alpha$ as the induction hypothesis.
Applying \eqref{eq:rho-recursion} and Lemma~\ref{lem:empty-intersection} again, we obtain
\begin{align*}
|\rho_\alpha| \le 
\frac{1}{1 - 2 \tau |\alpha|^2} \bigg( \p\{\alpha \cup \{(1, 2)\} \subseteq W\} + |\rho_\varnothing| \, P_{\alpha \varnothing} + \sum_{\beta \subsetneq \alpha} |\rho_\beta| \, P_{\alpha \beta} \bigg) ,
\end{align*}
where $\rho_\beta = 0$ if either $1$ or $2$ is not in $V(\beta)$ by Lemma~\ref{lem:zero-cumulant}. 
We then apply \eqref{eq:x-exp-bd-1} and \eqref{eq:x-exp-bd-2} for $\beta$ such that $1, 2 \in V(\beta)$ to obtain 
\begin{align*}
|\rho_\alpha| &\le \frac{1}{1 - 2 \tau |\alpha|^2} \bigg( (1+\tau) \, (\tau \, |V(\alpha)|)^{|V(\alpha)|-1} + \sum_{\beta \subsetneq \alpha} |\rho_\beta| \, (\tau \, |V(\alpha)|)^{|V(\alpha)| - |V(\beta)|} \bigg) .
\end{align*}
Then, by the induction hypothesis $|\rho_\beta| \le (1 + \tau |\beta|^4) \, (|\beta| + 1)^{|\beta|} \, (\tau \, |V(\beta)|)^{|V(\beta)|-1}$ together with the assumption $\tau |\alpha|^4 \le 0.1$, we see that 
\begin{align*}
|\rho_\alpha| 
&\le 2 (\tau \, |V(\alpha)|)^{|V(\alpha)|-1} + \frac{1}{1 - 2 \tau |\alpha|^2} \cdot \sum_{\beta \subsetneq \alpha} (1 + \tau |\beta|^4) \, (|\beta| + 1)^{|\beta|} (\tau \, |V(\alpha)|)^{|V(\alpha)| - 1} \\
&= (\tau \, |V(\alpha)|)^{|V(\alpha)|-1} \, \bigg( 2 + \frac{1}{1 - 2 \tau |\alpha|^2} \cdot \sum_{\beta \,:\, \varnothing \subsetneq \beta \subsetneq \alpha} (1 + \tau |\beta|^4) \, (|\beta| + 1)^{|\beta|} \bigg) .
\end{align*}
Finally, since $\tau |\alpha|^4 \le 0.1$ and $|\beta| \le |\alpha|-1$, we have $\frac{1}{1 - 2 \tau |\alpha|^2} (1 + \tau |\beta|^4) \le (1 + \tau |\alpha|^4)$, and then
\begin{align*}
2 + \frac{1}{1 - 2 \tau |\alpha|^2} \sum_{\beta \,:\, \varnothing \subsetneq \beta \subsetneq \alpha} (1 + \tau |\beta|^4) \, (|\beta| + 1)^{|\beta|} 
&\le 2 + (1 + \tau |\alpha|^4) \sum_{i=1}^{|\alpha|-1} \sum_{\substack{\beta \,:\, \beta \subsetneq \alpha , \, |\beta| = i}} (i+1)^i \\
&= 2 + (1 + \tau |\alpha|^4) \sum_{i=1}^{|\alpha|-1} \binom{|\alpha|}{i} (i+1)^i \\
&\le (1 + \tau |\alpha|^4) \sum_{i=0}^{|\alpha|} \binom{|\alpha|}{i} |\alpha|^i \\
&= (1 + \tau |\alpha|^4) \, (|\alpha| + 1)^{|\alpha|}.
\end{align*}
Combining the above two displays finishes the induction.
\end{proof}

\begin{proposition}
\label{prop:corr-bound}
Recall \eqref{eq:corr-rho-bound} and \eqref{eq:rho-recursion}.
Suppose that $\tau D^4 \le 0.1$. 
We have:
\begin{itemize}
\item
If $n \tau^2 (D+1)^2 Q^{4 \sqrt{D}} \le 1/2$ where $Q := \max\{\lambda (D+1)^2, 1\}$, then
$$
\Corr_{\le D}^2 \le \tau^2 \big( 1+\lambda^2 + 4 n \tau^2 (D+1)^5 Q^{4 \sqrt{D}} \big) .
$$

\item
If $\lambda^2 (D+1)^4 M \le 1/2$ where $M := \max\{n \tau^2 (D+1)^2, 1\}$, then 
$$
\Corr_{\le D}^2 \le \tau^2 \big( 1 + 4 \lambda^2 (D+1)^{7} M \big).
$$
\end{itemize}
\end{proposition}

\begin{proof}
To ease the notation, we consider $D$ such that $\sqrt{D}/2$ is an integer; the proof can be easily adapted to the general case by using floors $\lfloor \cdot \rfloor$ or ceilings $\lceil \cdot \rceil$. 

To bound $\Corr_{\le D}^2$, we apply \eqref{eq:corr-rho-bound} and \eqref{eq:rho-recursion}.
By Lemma~\ref{lem:zero-cumulant}, it suffices to consider connected graphs $\alpha \subseteq \binom{[n]}{2}$ such that $1,2 \in V(\alpha)$, for otherwise $\rho_\alpha = 0$.
In the sequel, we focus on such $\alpha$ but suppress the conditions for brevity. 
Note that there is only one such $\alpha$ with $|\alpha|=1$, i.e., the graph $\alpha^*$ consisting of a single edge $(1,2)$. 
For other graphs, we have $|\alpha| \ge 2$ and $|V(\alpha)| \ge 3$. 
Since the graph $\alpha$ is connected, we have $|V(\alpha)| \le |\alpha| + 1$.
Applying \eqref{eq:corr-rho-bound}, \eqref{eq:rho-recursion}, Lemma~\ref{lem:kappa-bd}, and the assumption $\tau D^4 \le 0.1$, we obtain
\begin{align*}
\Corr_{\le D}^2 
&\le \rho_\varnothing^2 + \lambda^2 \rho_{\alpha^*}^2 + \sum_{\alpha \,:\, 2 \le |\alpha| \le D} \lambda^{2 |\alpha|}  \, (1 + \tau |\alpha|^4)^2 \, (|\alpha|+1)^{2 |\alpha|} \, (\tau \, |V(\alpha)|)^{2 |V(\alpha)|-2}\\
&\le \tau^2 + \lambda^2 \tau^2 + 2 \sum_{\ell=2}^{D} \sum_{v=3}^{D+1} \sum_{\substack{\alpha \,:\, |\alpha| = \ell , \, |V(\alpha)| = v}} \\ 
&\qquad \qquad \qquad \qquad \qquad \big( \lambda (D+1) \big)^{2 \ell} \, \big( \tau(D+1) \big)^{2v-2} \, \bbone\{ v -1 \le \ell \le v^2/2 \} .
\end{align*}
Since $1, 2 \in V(\alpha)$, there are at most $\binom{n}{v-2} \, \Big[ \binom{\binom{v}{2}}{\ell} \land 2^{\binom{v}{2}} \Big]$ graphs $\alpha$ with $|V(\alpha)| = v$ and $|\alpha| = \ell$. 
By the bounds $\binom{n}{v-2} \le n^{v-2}$, $\binom{\binom{v}{2}}{\ell} \le \binom{v}{2}^\ell \le (v-1)^{2 \ell} \le (D+1)^{2 \ell}$, and $2^{\binom{v}{2}} \le 2^{v^2}$, it follows that
\begin{align}
\Corr_{\le D}^2 
&\le \tau^2 (1+\lambda^2) + 2 \tau^2 (D+1)^2 \cdot \notag \\ 
& \sum_{\ell=2}^{D} \sum_{v=3}^{D+1} \Big[ (D+1)^{2\ell} \land 2^{v^2} \Big] \big( \lambda (D+1) \big)^{2 \ell} \, \big( n \tau^2 (D+1)^2 \big)^{v-2} \, \bbone\{ v -1 \le \ell \le v^2/2 \} .
\label{eq:corr-bound-intermediate}
\end{align}
Let us consider two cases.
\paragraph{Case 1:}
We bound the summation in \eqref{eq:corr-bound-intermediate} by splitting it into the following terms according to the value of $v$:
\begin{subequations}
\begin{align}
\Corr_{\le D}^2 
\le \tau^2 (1+\lambda^2) &+ 2 \tau^2 (D+1)^2 \cdot \notag \\ 
&\bigg[ \sum_{v=3}^{\sqrt{D}} \sum_{\ell=2}^{D} 2^{v^2} \big( \lambda (D+1) \big)^{2 \ell} \, \big( n \tau^2 (D+1)^2 \big)^{v-2} \, \bbone\{ \ell \le v^2/2 \} 
\label{eq:corr-term-1} \\
&+ \sum_{\ell=2}^{D} \sum_{v=\sqrt{D}+1}^{D+1} (D+1)^{2\ell} \big( \lambda (D+1) \big)^{2 \ell} \, \big( n \tau^2 (D+1)^2 \big)^{v-2} \bigg] .
\label{eq:corr-term-2}
\end{align}
\end{subequations}
Recall that $Q := \max\{\lambda (D+1)^2, 1\}$. 
Moreover, by assumption,
$$
Q^{\sqrt{D}} \, n \tau^2 (D+1)^2 \le 1/2 .
$$
For $v \le \sqrt{D}$, we have $2 \ell \le v^2 \le \sqrt{D} \, v$, so the sum in \eqref{eq:corr-term-1} is bounded by 
\begin{align*}
\sum_{v=3}^{\sqrt{D}} D Q^{\sqrt{D} \, v} \, \big( n \tau^2 (D+1)^2 \big)^{v-2} 
&= D Q^{2 \sqrt{D}} \, \sum_{v=3}^{\sqrt{D}}  \Big( Q^{\sqrt{D}} \, n \tau^2 (D+1)^2 \Big)^{v-2} \\
&\le 2 D Q^{2 \sqrt{D}} Q^{\sqrt{D}} \, n \tau^2 (D+1)^2 . 
\end{align*}
Next, $n \tau^2 (D+1)^2 \le 1/2$ by assumption, so the sum in \eqref{eq:corr-term-2} is bounded by 
\begin{align*}
\sum_{\ell=2}^{D} \sum_{v=\sqrt{D}+1}^{D+1} Q^{2 \ell} \, \big( n \tau^2 (D+1)^2 \big)^{v-2} 
&\le 2 D Q^{2D} \big( n \tau^2 (D+1)^2 \big)^{\sqrt{D}-1} \\
&\le 2 D \big( Q^{4 \sqrt{D}} n \tau^2 (D+1)^2 \big)^{\sqrt{D}/2} \\
&\le 2 D Q^{4 \sqrt{D}} n \tau^2 (D+1)^2 ,
\end{align*}
where the last step holds because $Q^{4 \sqrt{D}} n \tau^2 (D+1)^2 \le 1/2$.
Plugging the above two bounds into \eqref{eq:corr-term-1} and \eqref{eq:corr-term-2} respectively, we complete the proof.

\paragraph{Case 2:}
Continuing from \eqref{eq:corr-bound-intermediate}, we have
\begin{align*}
\Corr_{\le D}^2 
\le \tau^2 (1+\lambda^2) + 2 \tau^2 (D+1)^2
\sum_{\ell=2}^{D} \sum_{v=3}^{\ell+2} \big( \lambda (D+1)^2 \big)^{2 \ell} \, \big( n \tau^2 (D+1)^2 \big)^{v-2} .
\end{align*}
Recall that $M = \max\{n \tau^2 (D+1)^2, 1\}$ and $\lambda^2 (D+1)^4 M \le 1/2$. 
We conclude that 
\begin{align*}
\Corr_{\le D}^2 
&\le \tau^2 (1+\lambda^2) + 2 \tau^2 (D+1)^2 \sum_{\ell=2}^{D} \big( \lambda^2 (D+1)^4 \big)^\ell \, M^\ell \ell \\
&\le \tau^2 (1+\lambda^2) + 2 \tau^2 (D+1)^2 \cdot 2 \lambda^2 (D+1)^4 M D \\
&\le \tau^2 \big( 1 + 4 (D+1)^7 \lambda^2 M \big) ,
\end{align*}
finishing the proof.
\end{proof}

We now prove Theorem~\ref{thm:recover-lower}.

\begin{proof}[Proof of Theorem~\ref{thm:recover-lower}]
It suffices to apply Proposition~\ref{prop:corr-bound} to bound $\Corr_{\le D}^2$. 
Consider two cases:
\begin{itemize}
\item
If $(\log n)^{-100} \le \lambda \le O(1)$, then $n \tau^2 \le n^{-\delta/2}$ by \eqref{eq:recover-lower-cond}.
We now apply the first statement of Proposition~\ref{prop:corr-bound}.
Since $D = o \Big( \big(\frac{\log n}{\log \log n}\big)^2 \Big)$, we have $Q = \max\{\lambda (D+1)^2, 1\} = \tilde O(1)$ and
$n \tau^2 (D+1)^2 Q^{4 \sqrt{D}} 
\le n \tau^2 \cdot n^{o(1)} \le 1/2$.
It follows that
$$
\Corr_{\le D}^2 \le \tau^2 \big( 1+\lambda^2 + 4 n \tau^2 (D+1)^5 Q^{4 \sqrt{D}} \big) 
\le \tau^2 (1 + \lambda^2 + n \tau^2 \cdot n^{o(1)})
= O(\tau^2) .
$$

\item
Next, suppose that $\lambda \le (\log n)^{-100}$, $n \tau^2 \lambda^2 \le n^{-\delta}$, and $D \le (\log n)^{10}$, which hold by \eqref{eq:recover-lower-cond}.
We apply the second statement of Proposition~\ref{prop:corr-bound} in each of the following two subcases:
\begin{itemize}
\item 
If $\tau \le n^{-1/2}$, then $M = \max\{n \tau^2 (D+1)^2, 1\} \le (D+1)^2$ and $\lambda^2 (D+1)^4 M \le 1/2$.
Therefore, $\Corr_{\le D}^2 \le \tau^2 \big( 1 + 4 \lambda^2 (D+1)^{7} M \big) = O(\tau^2).$

\item
If $\tau > n^{-1/2}$, then $M = n \tau^2 (D+1)^2$ and $\lambda^2 (D+1)^4 M = n \tau^2 \lambda^2 (D+1)^6 \le 1/2$.
We again obtain $\Corr_{\le D}^2 \le \tau^2 \big( 1 + 4 \lambda^2 (D+1)^{7} M \big) = O(\tau^2).$
\end{itemize}
\end{itemize}
Combining the above cases, we conclude that $\Corr_{\le D} = O(\tau)$ if \eqref{eq:recover-lower-cond} holds.
This completes the proof once we recall Definition~\ref{def:weak-recovery} and that $\E_{\cP}[\chi] = \tau$.
\end{proof}

\subsection{Recovery upper bound}
\label{sec:recover-upper-proof}

For brevity, write $T = T(A)$ in the sequel.
We let $i_0 := 1$ and $i_{\ell+1} := 2$, so that a length-$(\ell+1)$ self-avoiding walk in consideration is through vertices $i_j$ for $j = 0,1,\dots,\ell+1$.
Hence we can rewrite \eqref{eq:def-stat-t} as
\begin{equation}
T = \sum_{\substack{3 \le i_1,\ldots,i_\ell \le n \\i_1\neq \cdots \neq i_\ell}} \prod_{j=0}^{\ell}  \tilde A_{i_j i_{j+1}} .
\label{eq:redefine-t}
\end{equation}
We assume $2 \tau (\ell+1) \le 1$ in the rest of this section.

\begin{lemma}\label{lem:centered Y property}
Let $\tilde A_{ij}$ be defined by \eqref{eq:def-a-tilde} and $\lambda$ be defined by \eqref{eq:def-lambda}.
We have
    \begin{itemize}
        \item $\E\kr{\tilde{A}_{ij}| z_i, z_j} = \lambda \cdot \bbone\{\dist(z_i,z_j) \le \tau/2\}$;
        \item $\E\kr{\tilde{A}_{ij}^2|z_i,z_j} \le p/q$.
    \end{itemize}
\end{lemma}

\begin{proof}
The first statement is obvious in view of \eqref{eq:def-a-tilde} and \eqref{eq:def-lambda}.
For the second statement, note that if $\dist(z_i,z_j) > \tau/2$, then $\E\kr{\tilde{A}_{ij}^2|z_i,z_j} = 1$, and if $\dist(z_i,z_j) \le \tau/2$, then
$$
\E\kr{\tilde{A}_{ij}^2|z_i,z_j}
= \frac{p(1-q)^2 + (1-p)q^2}{q(1-q)}
\le \frac{p(1-q) + q^2}{q} \le \frac{p}{q}.
$$
\end{proof}

\begin{proposition}\label{prop:recovery-expectation}
If $\dist(z_1,z_2) > \frac{(\ell+1)\tau}{2}$, then $\E[T \cond z_1, z_2] = 0$. 
If $\dist(z_1,z_2) \le \frac{(\ell+1)\tau}{2}$, then 
    \begin{equation}
        \E\kr{T\cond z_1,z_2} = \binom{n-2}{\ell} \tau^{\ell} \lambda^{\ell+1}\int_{\frac{\ell}{2}+\frac{\dist(z_1,z_2)}{\tau}-\frac{1}{2}}^{\frac{\ell}{2}+\frac{\dist(z_1,z_2)}{\tau}+\frac{1}{2}}f_\ell(t)\rmd t ,
        \label{eq:cond-exp-t-z2}
    \end{equation}
    where $\PDFIH_\ell(x)$ is the probability density function of the Irwin-Hall distribution with parameter $\ell$, i.e.,
    \begin{equation}
        \PDFIH_\ell(x) = \frac{1}{(\ell - 1)!} \sum_{k=0}^{\lfloor x \rfloor} (-1)^k \binom{\ell}{k} (x - k)^{\ell-1} \quad \text{ for } x \in [0,\ell], 
        \label{eq:def-ih-pdf}
    \end{equation}
    and $\PDFIH_\ell(x) = 0$ otherwise. 
Moreover, $u \mapsto \E[T \cond \dist(z_1,z_2) = u]$ is a decreasing function on $[0, \frac{(\ell+1)\tau}{2}]$.
\end{proposition}

\begin{proof}
Throughout this proof, we condition on $z_1$ and $z_2$, and use $\E$ and $\p$ to denote the conditional expectation and conditional probability respectively. 
By \eqref{eq:redefine-t} and the independence of $(A_{i_j i_{j+1}})_{j=0}^\ell$ conditional on $(z_{i_j})_{j=0}^{\ell+1}$, we have
$$
\E\kr{T} = \sum_{\substack{3 \le i_1,\ldots,i_\ell \le n \\i_1\neq \cdots \neq i_\ell}} \E \Bigg[ \prod_{j=0}^{\ell}\E \kr{ \tilde A_{i_ji_{j+1}} \cond (z_{i_s})_{s=0}^{\ell+1}} \Bigg] .
$$
Applying the first statement of Lemma~\ref{lem:centered Y property}, we then obtain
\begin{equation}
\E\kr{T} = \lambda^{\ell+1} \sum_{\substack{3 \le i_1,\cdots,i_\ell \le n \\i_1\neq \cdots \neq i_\ell}} \p \left\{ \dist(z_{i_s}, z_{i_{s+1}}) \le \tau/2 \text{ for all } s \in [\ell]  \right\} .
\label{eq:t-cond-exp-inter}
\end{equation}
If $\dist(z_{i_s}, z_{i_{s+1}}) \le \tau/2$ for all $s \in [\ell]$, then $\dist(z_1,z_2) \le (\ell+1)\tau/2$ by the triangle inequality.
Therefore, we see that
$\E[T ] = 0$ if $\dist(z_1,z_2) > (\ell+1)\tau/2$.

Next, suppose that $\dist(z_1,z_2) \le (\ell+1)\tau/2$.
Fix vertices $i_1, \dots, i_\ell$ and define
\begin{equation*}
E_s := \br{ \dfrak(z_{i_s},z_{i_{s+1}})\leq \frac{\tau}{2}}, \qquad
\EventExpect = \bigcap_{s=0}^{\ell} E_s,
\end{equation*}
where we suppress the dependency on $i_1, \dots, i_\ell$ for brevity.
We now compute $\p\{\cE\}$, i.e., the probability in \eqref{eq:t-cond-exp-inter}.
Let us write
\begin{equation}
\Pb\br{\EventExpect} 
= \prod_{s=0}^{\ell} \p\bigg\{ E_s \; \bigg| \; \bigcap_{j=0}^{s-1} E_j \bigg\} .
\label{eq:prob-product}
\end{equation}
Since $(z_{i_s})_{s=1}^\ell$ are i.i.d.\ uniform random variables in $[0,1]$, it is not hard to see that
\begin{equation}
\p\bigg\{ E_s \; \bigg| \; \bigcap_{j=0}^{s-1} E_j \bigg\} = \tau  \quad \text{ for } 0 \le s \le \ell-1.
\label{eq:prod-s-l-1}
\end{equation}

It remains to compute the conditional probability $\Pb\br{E_{\ell}\cond \bigcap_{j=0}^{\ell-1} E_j}$.
For any $0 \le s \le \ell-1$, conditional on any realization of $z_{i_0}, z_{i_1}, \dots, z_{i_s}$ and the event $E_s$, 
the random variable $z_{i_{s+1}} - z_{i_s}$ is uniform $[-\tau/2, \tau/2]$.
Crucially, this distribution does not depend on $z_{i_0}, z_{i_1}, \dots, z_{i_s}$.
Applying this argument for $s = 0,1, \dots, \ell-1$, we see that conditional on $\bigcap_{s = 0}^{\ell-1} E_s$, the random variables $z_{i_1} - z_{i_0}$, $z_{i_2} - z_{i_1}$, \dots, $z_{i_\ell} - z_{i_{\ell-1}}$ are i.i.d.\ and uniform in $[-\tau/2, \tau/2]$.
We can write
$$
z_{i_\ell} = z_{i_0} + \tau I_\ell - \frac{\ell \tau}{2} , \quad \text{ where }
I_\ell := \sum_{s=0}^{\ell-1} \Big( \frac{z_{i_{s+1}} - z_{i_s}}{\tau} + \frac 12 \Big) .
$$
Since $I_r$ is a sum of $\ell$ i.i.d.\ uniform random variables in $[0,1]$, it has the Irwin--Hall distribution with parameter $\ell$ (see, e.g., \cite{johnson1995continuous}).
Moreover, since $i_0 = 1$ and $i_{\ell+1} = 2$, the event $E_\ell$ occurs if and only if $\dist(z_2, z_1 + \tau I_\ell - \frac{\ell \tau}{2}) \le \tau/2$, i.e.,
$$
\frac{z_2 - z_1}{\tau} + \frac{\ell}{2} - \frac 12 \le I_\ell \le \frac{z_2 - z_1}{\tau} + \frac{\ell}{2} + \frac 12 .
$$
Let $\PDFIH_\ell (x)$ be the PDF of the Irwin--Hall distribution with parameter $\ell$.
Then
\begin{equation}
\p \bigg\{ E_{\ell}\; \bigg| \; \bigcap_{j=0}^{\ell-1} E_j \bigg\} = \int_{\frac{\ell}{2}+\frac{z_2-z_1}{\tau}-\frac{1}{2}}^{\frac{\ell}{2}+\frac{z_2-z_1}{\tau}+\frac{1}{2}} \PDFIH_\ell (t) \rmd t.
\label{eq:prob-e-l}
    \end{equation}

Plugging \eqref{eq:prod-s-l-1} and \eqref{eq:prob-e-l} into \eqref{eq:prob-product} and then combining the result with \eqref{eq:t-cond-exp-inter}, we obtain    \begin{equation*}
\E\kr{T} = \lambda^{\ell+1} \binom{n-2}{\ell} \, \tau^{\ell} \int_{\frac{\ell}{2}+\frac{z_2 - z_1}{\tau}-\frac{1}{2}}^{\frac{\ell}{2}+\frac{z_2 - z_1}{\tau}+\frac{1}{2}}f_\ell(t)\rmd t, 
\end{equation*}
which is almost \eqref{eq:cond-exp-t-z2}.
It remains to show that the above quantity is an even function in $u := z_2 - z_1$ and decreasing for $u \in [0, \frac{(\ell+1)\tau}{2}]$.
Its derivative as a function of $u$ is proportional to
\begin{equation}
f_\ell \Big( \frac{\ell}{2}+\frac{u}{\tau}+\frac{1}{2} \Big) - f_\ell \Big( \frac{\ell}{2}+\frac{u}{\tau}-\frac{1}{2} \Big) .
\label{eq:deri-diff}
\end{equation}
The PDF $f_\ell(t)$ is symmetric around $\ell/2$, increasing on $[0,\ell/2]$, decreasing on $[\ell/2, \ell]$, and zero outside $[0,\ell]$ (and the monotonicity of $f_\ell(t)$ on $[0,\ell/2]$ and $[\ell/2, \ell]$ is strict if $\ell > 1$).
Hence, the difference in \eqref{eq:deri-diff} is an odd function in $u$; it is positive if $u \in [-\frac{(\ell+1)\tau}{2}, 0]$ and negative if $u \in [0, \frac{(\ell+1)\tau}{2}]$.
Consequently, $\E[T]$ is an even function in $u = z_2 - z_1$, and it is increasing on $[-\frac{(\ell+1)\tau}{2}, 0]$ and decreasing on $[0, \frac{(\ell+1)\tau}{2}]$, proving the last statement.
\end{proof}

\begin{lemma}
\label{lem:separation lemma}
    Let $\epsilon \in (0, \tau/2)$.
    There is a constant $\constGap > 0$ depending only on $\ell$ such that
\begin{equation}
        \Delta(\epsilon):=\E\kr{T\cond \dfrak(z_1,z_2) = \frac{\tau}{2}} - \E\kr{T\cond \dfrak(z_1,z_2)=\frac{\tau}{2}+\epsilon} \ge \constGap \, n^\ell \epsilon \, \tau^{\ell-1} \lambda^{\ell+1}.
        \label{eq:def-delta-epsilon}
    \end{equation}
\end{lemma}

\begin{proof}
It follows from Proposition~\ref{prop:recovery-expectation} that
\begin{align*}
\Delta(\epsilon) &= \pr{\int_{\frac{\ell}{2}}^{\frac{\ell}{2}+1} \PDFIH_\ell(t)\rmd t - \int_{\frac{\ell}{2}+\frac{\epsilon}{\tau}}^{\frac{\ell}{2}+1+\frac{\epsilon}{\tau}}\PDFIH_\ell(t)\rmd t} \binom{n-2}{\ell} \tau^{\ell} \lambda^{\ell+1} \\
&= \pr{\int_{\frac{\ell}{2}}^{\frac{\ell}{2}+\frac{\epsilon}{\tau}}\PDFIH_\ell(t)\rmd t 
-\int_{\frac{\ell}{2}+1}^{\frac{\ell}{2}+1+\frac{\epsilon}{\tau}}\PDFIH_\ell(t)\rmd t} \binom{n-2}{\ell} \tau^{\ell} \lambda^{\ell+1}.
\end{align*}
By the mean value theorem, there exists $\xi_1 \in \pr{\frac{\ell}{2},\frac{\ell}{2}+\frac{\epsilon}{\tau}}$ and $\xi_2 \in \pr{\frac{\ell}{2}+1,\frac{\ell}{2}+1+\frac{\epsilon}{\tau}}$
such that
    \begin{equation}
        \Delta(\epsilon) = \frac{\epsilon}{\tau}\pr{\PDFIH_{\ell}(\xi_1)-\PDFIH_{\ell}(\xi_2)} \binom{n-2}{\ell} \tau^{\ell} \lambda^{\ell+1}.
        \label{eq:delta-eps-gap}
    \end{equation}
If $\ell = 1$, then $\PDFIH_{\ell}(\xi_1)-\PDFIH_{\ell}(\xi_2) = 1$ as $\eps/\tau \in (0,1/2)$;
if $\ell > 1$, then $\PDFIH_\ell(x)$ is strictly decreasing for $x \in [\ell/2, \ell]$, so $\PDFIH_{\ell}(\xi_1)-\PDFIH_{\ell}(\xi_2) \ge c'_\ell$ for a constant $c'_\ell > 0$.
The conclusion follows from \eqref{eq:delta-eps-gap}.
\end{proof}

Recall that we identify an edge set $\alpha \subseteq \binom{[n]}{2}$ with the graph induced by $\alpha$, and $V(\alpha) \subseteq [n]$ denotes the vertex set of $\alpha$.
Recall \eqref{eq:def-saw}.
For $\alpha,\beta\in \SAW$, we consider the graph $\alpha \triangle \beta$ and introduce the following notation which will be used in the rest of this section:
\begin{subequations}
\label{eq:def-evc}
\begin{align}
\esf &:= |\alpha \triangle \beta| , \\
\vsf &:= |V(\alpha \triangle \beta)| , \\
\csf &:= \text{number of connected components of } \alpha \triangle \beta .
\end{align}
\end{subequations}

\begin{lemma}\label{lem:variance summand bound}
For $\alpha,\beta\in \SAW$, let $\esf$, $\vsf$, and $\csf$ be defined in \eqref{eq:def-evc}. 
Recall \eqref{eq:def-a-alpha}.
We have
\begin{align*}
\E \left[\tilde{A}_{\alpha\cap \beta}^2 \tilde{A}_{\alpha\triangle \beta} \mid z_1, z_2 \right]
\le 
\begin{cases}
(p/q)^{\ell+1} 
& \text{ if } \alpha \triangle \beta = \varnothing ,\\
(p/q)^{\ell+1-\esf/2} \,
\lambda^{\esf} (\ell\tau)^{\vsf-\csf-1} & \text{ if } \alpha \triangle \beta \ne \varnothing .
\end{cases}
\end{align*}
\end{lemma}

\begin{proof}
For brevity, write $z = \{z_i : i \in V(\alpha \cup \beta)\}$,
and let $\E$ and $\p$ be the expectation and probability conditional on $z_1, z_2$ in the proof.
By the independence of $(\tilde{A}_{ij})_{(i,j) \in \alpha\cap \beta}$ and $(\tilde{A}_{ij})_{(i,j) \in \alpha\triangle \beta}$ conditional on $z$, we have
\begin{equation*}
\E\kr{\tilde{A}_{\alpha\cap \beta}^2\tilde{A}_{\alpha\triangle \beta}}
    = \E \bigg[ \prod_{(i,j) \in \alpha \cap \beta} \E\kr{\tilde{A}_{ij}^2\cond z} \cdot \prod_{(i,j) \in \alpha \triangle \beta} \E\kr{\tilde{A}_{ij}\cond z} \bigg].
\end{equation*}
It then follows from Lemma~\ref{lem:centered Y property} that
\begin{equation}
    \E\kr{\tilde{A}_{\alpha\cap \beta}^2\tilde{A}_{\alpha\triangle \beta}}
    \le 
    (p/q)^{|\alpha \cap \beta|} 
    \, \lambda^{\esf} \cdot \p\{ \dist(z_i,z_j) \le \tau/2 \text{ for all } (i,j) \in \alpha \triangle \beta \} .
    \label{eq:exp-a-tilde-sq}
\end{equation}
If $\alpha \triangle \beta = \varnothing$, then $\esf = 0$ and $|\alpha \cap \beta| = \ell+1$, so the first bound of the lemma follows.
For the second bound where $\alpha \triangle \beta \ne \varnothing$, note that 
$|\alpha \cap \beta| = \frac 12 (|\alpha| + |\beta| - |\alpha \triangle \beta|) = \frac 12 (2 \ell + 2 - \esf)$.
Hence, it remains to bound the probability in \eqref{eq:exp-a-tilde-sq} by $(\ell \tau)^{\vsf - \csf - 1}$.

Suppose that $\dist(z_i,z_j) \le \tau/2$ for all $(i,j) \in \alpha \triangle \beta$.
Choose vertices $j_1, \dots, j_{\csf} \in V(\alpha \triangle \beta)$, one from each of the $\csf$ connected components of $\alpha \triangle \beta$; in particular, if $1 \in V(\alpha \triangle \beta)$, we choose $j_1 = 1$. For every $i \in V(\alpha \triangle \beta) \setminus \{1,2\}$, there is a path of length at most $\ell$ from vertex $i$ to vertex $j_{s_i}$ for some $s_i \in [\csf]$ such that the path lies entirely in (the $s$th connected component of) $\alpha \triangle \beta$.
It follows that $\dist(z_i, z_{j_{s_i}}) \le \ell \tau/2$.
Therefore, 
\begin{align}
&\p\{ \dist(z_i,z_j) \le \tau/2 \text{ for all } (i,j) \in \alpha \triangle \beta \} \notag \\
&\le \p \Big\{ \dist(z_i,z_{j_{s_i}}) \le \ell \tau/2 \text{ for all } i \in V(\alpha \triangle \beta) \setminus (\{j_1, \dots, j_{\csf}\}\cup\{1,2\}) \Big\}
\le (\ell \tau)^{\vsf - \csf - 1} , \label{eq:ltau2vc1}
\end{align}
since the random variables $\{z_i : i \in V(\alpha \triangle \beta) \setminus (\{j_1, \dots, j_{\csf}\}\cup\{1,2\}) \}$ are i.i.d.\ uniform in $[0,1]$ conditional on any realization of $z_{j_1}, \dots, z_{j_{\csf}}, z_1, z_2$. 
Plugging the above bound into \eqref{eq:exp-a-tilde-sq} completes the proof.
\end{proof}

We now state a slightly improved version of the above lemma when $z_1$ and $z_2$ are far apart.

\begin{lemma}\label{lem:variance summand bound ver 2}
In the setting of the above lemma, if $\dist(z_1,z_2) > \frac{(\ell+1) \tau}{2}$ and $\alpha \triangle \beta \ne \varnothing$, we have
\begin{align*}
\E \left[\tilde{A}_{\alpha\cap \beta}^2 \tilde{A}_{\alpha\triangle \beta} \mid z_1, z_2 \right]
\le (p/q)^{\ell+1-\esf/2} \,
\lambda^{\esf} (\ell\tau)^{\vsf-\csf} .
\end{align*}
\end{lemma}

\begin{proof}
The only difference from the above lemma is that we now have $(\ell\tau)^{\vsf-\csf}$ instead of $(\ell\tau)^{\vsf-\csf-1}$.
This difference originates from \eqref{eq:ltau2vc1}.
Recall that we suppose $\dist(z_i,z_j) \le \tau/2$ for all $(i,j) \in \alpha \triangle \beta$.
However, since $\dist(z_1,z_2) > \frac{(\ell+1)\tau}{2}$, vertices $1$ and $2$ cannot be in the same connected components of $\alpha \triangle \beta$.
Therefore, when selecting the vertices $j_1, \dots, j_{\csf}$, we can choose $j_1 = 1$ and $j_2 = 2$ without loss of generality.
Then \eqref{eq:ltau2vc1} becomes
\begin{align*}
&\p\{ \dist(z_i,z_j) \le \tau/2 \text{ for all } (i,j) \in \alpha \triangle \beta \} \notag \\
&\le \p \Big\{ \dist(z_i,z_{j_{s_i}}) \le \ell \tau/2 \text{ for all } i \in V(\alpha \triangle \beta) \setminus \{j_1, \dots, j_{\csf}\} \Big\}
\le (\ell \tau)^{\vsf - \csf} , 
\end{align*}
thereby improving the bound by a factor $\ell\tau$.
\end{proof}

\begin{lemma}\label{lem:vcex}
For $\alpha,\beta\in \SAW$, let $\esf$, $\vsf$, and $\csf$ be defined in \eqref{eq:def-evc}.
We have 
$$
\abs{V(\alpha\cup \beta)} \leq \vsf -\frac{1}{2}\esf - \csf + \ell + 2 .
$$
\end{lemma}

\begin{proof}
Note that the graph $\alpha \cup \beta$ is the disjoint union of $\alpha \triangle \beta$ and $\alpha \cap \beta$. 
To bound the number of vertices of $\alpha \cup \beta$, we start from the graph $K = \alpha \triangle \beta$, which has $\vsf$ vertices, and then sequentially add vertices and edges of $\alpha \cap \beta$ to $K$ until we eventually reach $K = \alpha\cup \beta$.
Hence, $|V(\alpha \cup \beta)|$ will be bounded by $\vsf$ plus the total number of vertices we add in this procedure.

To be more precise, at each step, we choose an edge $(i,j)$ of $\alpha \cap \beta$ that has not yet been added to the current $K$, such that $i \in V(K)$.
Such an edge exists because the eventual graph $\alpha \cup \beta$ is connected.
Now we add $(i,j)$ to $K$, and there are two cases: $j \in V(K)$ or $j \notin V(K)$.
If $j \in V(K)$, then $|V(K)|$ does not increase; if $j \notin V(K)$, then $|V(K)|$ increases by $1$.
Moreover, the number of connected components of $K$ may decrease by $1$ if $j \in V(K)$ (when $(i,j)$ connects two components); the number of connected components of $K$ will not decrease if $j \notin V(K)$.
Since the number of connected components of $K$ decreases from $\csf$ to $1$ in the entire procedure, the first case must occur at least $\csf - 1$ times, so $|V(K)|$ does not increase in at least $\csf - 1$ steps.
Since there are $|\alpha \cap \beta|$ steps of adding an edge in total, the number of vertices added is at most $|\alpha \cap \beta| - \csf + 1$.
Therefore, we obtain
\begin{equation*}
\abs{V(\alpha \cup \beta)} \leq \vsf + \abs{\alpha \cap \beta} -\csf+1.
\end{equation*}
To complete the proof, it suffices to recall that $|\alpha| = |\beta| = \ell + 1$ so that $2\abs{\alpha\cap \beta}+\esf = 2\ell +2$. 
\end{proof}

\begin{lemma}
\label{lem:v-2-ell}
For $\alpha,\beta\in \SAW$, let $\vsf$ be defined in \eqref{eq:def-evc}.
Suppose $\alpha \cap \beta \ne \varnothing$.
Then we have $\vsf \le 2 \ell$.
\end{lemma}

\begin{proof}
Let $(i,j) \in \alpha \cap \beta$ where $i<j$.
If $i = 1$, then $1 \notin V(\alpha \triangle \beta)$ and $|V(\alpha \cup \beta)| \le 2 \ell + 1$.
We see that $\vsf \le |V(\alpha \cup \beta)| - 1 \le 2 \ell$.
The case $j = 2$ is similar.
In other cases where $i \ne 1$ and $j \ne 2$, we have $\vsf \le |V(\alpha \cup \beta)| \le 2 \ell$.
\end{proof}

\begin{lemma}\label{lem:vce-excess}
For $\alpha,\beta\in \SAW$, let $\esf$, $\vsf$, and $\csf$ be defined in \eqref{eq:def-evc}.
Suppose $\alpha \triangle \beta \ne \varnothing$.
Then the graph $\alpha \triangle \beta$ does not contain any dangling edge, i.e., an edge $(i,j)$ such that vertex $j$ is connected to only vertex $i$ in $\alpha \triangle \beta$.
As a result, we have 
$\csf + \esf - \vsf \ge 1.$
\end{lemma}

\begin{proof}
The quantity $\csf + \esf - \vsf$ is known as the excess of the graph $\alpha \triangle \beta$; it is always nonnegative and is zero only if $\alpha \triangle \beta$ is a forest.
Since a forest obviously contains a dangling edge, it remains to prove that $\alpha \triangle \beta$ does not contain a dangling edge. 

To see this, it is convenient to view $\alpha \cup \beta$ as a multigraph, which has even degree at each vertex.
Further, to obtain $\alpha \triangle \beta$ from $\alpha \cup \beta$, we delete all the double edges in $\alpha \cap \beta$, so $\alpha \triangle \beta$ also has even degree at each vertex.
As a result, $\alpha \triangle \beta$ does not contain a dangling edge.
\end{proof}

\begin{proposition}    \label{prop:recovery-variance}
There is a constant $\constVar > 0$ that depends only on $\ell$ such that 
$$
\Var(T\cond z_1, z_2) \leq \constVar \, 
\bigg[ n^\ell \left(\frac{p}{q}\right)^{\ell+1} + n^{2\ell - 1} \tau^{2 \ell - 2} \lambda^{2 \ell} \, \frac{p}{q} + n^{\ell + \frac 12} \tau \lambda^3 \left(\frac{p}{q}\right)^{\ell - \frac 12} \bigg] .
$$
Moreover, if $\dist(z_1,z_2) > \frac{(\ell+1) \tau}{2}$, then
$$
\Var(T\cond z_1, z_2) \leq \constVar \, 
\bigg[ n^\ell \left(\frac{p}{q}\right)^{\ell+1} + n^{2\ell - 1} \tau^{2 \ell - 1} \lambda^{2 \ell} \, \frac{p}{q} + n^{\ell + \frac 12} \tau^2 \lambda^3 \left(\frac{p}{q}\right)^{\ell - \frac 12} \bigg] .
$$
\end{proposition}

\begin{proof}
Throughout this proof, we condition on $z_1$ and $z_2$, and the notations $\E$, $\p$, and $\Var$ are all with respect to the conditional probability. 
By \eqref{eq:def-stat-t}, we have
\begin{equation}
\Var(T)
= \E[T^2]  - (\E T)^2 
= \sum_{\alpha,\beta\in \SAW} \E [\tilde{A}_\alpha \tilde{A}_\beta] - \E[\tilde{A}_\alpha] \cdot \E[\tilde{A}_\beta] .
\label{eq:var-first-eq}
\end{equation}
Note that $\ell + 2 \le |V(\alpha \cup \beta)| \le 2 \ell + 2$.

Let us first consider the extreme case $|V(\alpha \cup \beta)| = 2 \ell + 2$, where the two walks $\alpha$ and $\beta$ have disjoint edge sets and only common vertices $1$ and $2$.
By the independence of $(\tilde A_{ij})_{(i,j) \in \alpha \cup \beta}$ conditional on $(z_i)_{i \in V(\alpha \cup \beta)}$ and the first statement of Lemma~\ref{lem:centered Y property}, we have 
\begin{align*}
\E [\tilde{A}_\alpha \tilde{A}_\beta]
&= \E \bigg[ \prod_{(i,j) \in \alpha \cup \beta} \E[ \tilde A_{ij} \mid (z_i)_{i \in V(\alpha \cup \beta)} ] \bigg] \\
&= \lambda^{2 \ell + 2} \, \p\{ \dist(z_i, z_j) \le \tau/2 \text{ for all } (i,j) \in \alpha \cup \beta \} .
\end{align*}
Similarly,
$$
\E[\tilde{A}_\alpha] \cdot \E[\tilde{A}_\beta]
= \lambda^{2 \ell + 2} \, \p\{ \dist(z_i, z_j) \le \tau/2 \text{ for all } (i,j) \in \alpha \} \cdot \p\{ \dist(z_i, z_j) \le \tau/2 \text{ for all } (i,j) \in \beta \} .
$$
Recall that we already condition on $z_1$ and $z_2$, and the variables $\{z_i : i \in V(\alpha), \, i \ne 1,2\}$ and $\{z_i : i \in V(\beta), \, i \ne 1,2\}$ are independent.
Hence, the above two displays are equal, i.e.,
$$
\E [\tilde{A}_\alpha \tilde{A}_\beta]
- \E[\tilde{A}_\alpha] \cdot \E[\tilde{A}_\beta] = 0.
$$

In all other case where $\ell + 2 \le |V(\alpha \cup \beta)| \le 2 \ell + 1$, we have that $\alpha \cap \beta \ne \varnothing$.
The first statement of Lemma~\ref{lem:centered Y property} implies $\E[\tilde A_\alpha] , \E[\tilde A_\beta] \ge 0$.
Therefore, we conclude from \eqref{eq:var-first-eq} that
$$
\Var(T)
\le \sum_{v=\ell+2}^{2\ell+1} \sum_{\substack{\alpha,\beta\in\SAW,\\ |V(\alpha\cup \beta)| = \vunion}} \E [\tilde{A}_\alpha \tilde{A}_\beta] .
$$
If $|V(\alpha\cup \beta)| = \vunion$, there are $\binom{n-2}{v-2} \le n^{v-2}$ choices of the vertices in $V(\alpha \cup \beta) \setminus \{1,2\}$.
With the vertices of $\alpha \cup \beta$ fixed, the number of possible graphs $\alpha \cup \beta$ is bounded by a constant $C_1 = C_1(\ell) > 0$.
Moreover, we can write $\tilde{A}_\alpha \tilde{A}_\beta = \tilde{A}^2_{\alpha\cap \beta} \tilde{A}_{\alpha\triangle \beta}$.
It follows that
\begin{align*}
\Var(T) 
\le C_1 \sum_{v={\ell+2}}^{2\ell+1} n^{\vunion-2} \max_{\substack{\alpha,\beta \in \SAW,\\ \abs{V(\alpha\cup \beta)}=v}} \E [\tilde{A}^2_{\alpha\cap \beta} \tilde{A}_{\alpha\triangle \beta}] .
\end{align*}

Let $\esf$, $\vsf$, and $\csf$ 
be defined in \eqref{eq:def-evc}.
If $\alpha \triangle \beta = \varnothing$, then $|V(\alpha \cup \beta)| = \ell+2$; if $\alpha \triangle \beta \ne \varnothing$, then $|V(\alpha \cup \beta)| \ge \ell+3$.
Applying Lemmas~\ref{lem:variance summand bound} and~\ref{lem:vcex} together with the above bound on $\Var(T)$, we obtain
\begin{align*}
\Var(T) 
&\le C_1 \bigg[ n^\ell (p/q)^{\ell+1} + \sum_{v={\ell+3}}^{2\ell+1} \max_{\substack{\alpha,\beta \in \SAW,\\ \abs{V(\alpha\cup \beta)}=v}}  n^{\vsf - \esf/2 - \csf + \ell} (p/q)^{\ell + 1 - \esf/2} \, \lambda^{\esf} (\ell\tau)^{\vsf-\csf-1} \bigg] \\
& \le C_1 \bigg[ n^\ell (p/q)^{\ell+1} +  \ell  \max_{\substack{\alpha,\beta \in \SAW , \\ \ell + 3 \le |V(\alpha \cup \beta)| \le 2 \ell + 1}}  n^{\vsf - \esf/2 - \csf + \ell} (p/q)^{\ell + 1 - \esf/2} \, \lambda^{\esf} (\ell\tau)^{\vsf-\csf-1} \bigg] .
\end{align*}
It follows that, for a sufficiently large constant $C_2 = C_2(\ell) > 0$,
\begin{align}
\Var(T) 
\le C_2 \, n^\ell \left(\frac{p}{q}\right)^{\ell+1}  \Bigg[ 1 + \frac{1}{\tau} \max_{\substack{\alpha,\beta \in \SAW , \\ \ell + 3 \le |V(\alpha \cup \beta)| \le 2 \ell + 1}} \left( \frac{n\tau^2 \lambda^2 q}{p} \right)^{\frac{1}{2}(\vsf-\csf)} \pr{\frac{\lambda^2 q}{n p}}^{\frac{1}{2}(\csf+\esf-\vsf)} \Bigg] .
\label{eq:var-intermediate-bound}
\end{align}
To further control the above maximum, we consider the two factors:
\begin{itemize}
\item
Since $\lambda = \frac{p-q}{\sqrt{q(1-q)}}$, we have $\frac{\lambda^2 q}{n p} = \frac{(p-q)^2}{n p (1-q)} \le 1$.
In addition, $\csf + \esf - \vsf \ge 1$ by Lemma~\ref{lem:vce-excess}.
Hence, it holds that
$\pr{\frac{\lambda^2 q}{n p}}^{\frac{1}{2}(\csf+\esf-\vsf)} \le \pr{\frac{\lambda^2 q}{n p}}^{\frac{1}{2}}$.

\item
By Lemma~\ref{lem:v-2-ell} and $\csf \ge 1$, we have $\vsf - \csf \le 2 \ell - 1$.
By Lemma~\ref{lem:vce-excess}, $\alpha \triangle \beta$ does not contain any dangling edge, so every connected component of it has at least three vertices, and thus $\vsf - \csf \ge 2$.
It follows that
$\left( \frac{n\tau^2 \lambda^2 q}{p} \right)^{\frac{1}{2}(\vsf-\csf)} \le \left( \frac{n\tau^2 \lambda^2 q}{p} \right)^{\ell - \frac 12} \lor \left( \frac{n\tau^2 \lambda^2 q}{p} \right)$.
\end{itemize}
Combining these facts with the above bound on $\Var(T)$, we see that
\begin{align*}
\Var(T) 
&\le C_4 \, n^\ell \left(\frac{p}{q}\right)^{\ell+1}  \bigg[ 1 + \frac{1}{\tau} \left( \frac{n\tau^2 \lambda^2 q}{p} \right)^{\ell - \frac 12} \pr{\frac{\lambda^2 q}{n p}}^{\frac{1}{2}} + \frac{1}{\tau} \left( \frac{n\tau^2 \lambda^2 q}{p} \right) \pr{\frac{\lambda^2 q}{n p}}^{\frac{1}{2}} \bigg] \\
&\le C_4 \, \bigg[ n^\ell \left(\frac{p}{q}\right)^{\ell+1} + n^{2\ell - 1} \tau^{2 \ell - 2} \lambda^{2 \ell} \, \frac{p}{q} + n^{\ell + \frac 12} \tau \lambda^3 \left(\frac{p}{q}\right)^{\ell - \frac 12} \bigg] 
\end{align*}
for a constant $C_4 = C_4(\ell) > 0$.

Finally, if $\dist(z_1,z_2) > \frac{(\ell+1) \tau}{2}$, then the application of Lemma~\ref{lem:variance summand bound} can be replaced by Lemma~\ref{lem:variance summand bound ver 2} in the above proof, so that we gain a factor $\tau$ in the case $\alpha \triangle \beta \ne \varnothing$.
As a result, we do not have the factor $1/\tau$ before the max in \eqref{eq:var-intermediate-bound}.
Consequently, we gain a factor $\tau$ in the second and the third term of the eventual bound.
\end{proof}

We now prove Theorems~\ref{thm:recover-upper} and~\ref{thm:exact-recovery}.

\begin{proof}[Proof of Theorem~\ref{thm:recover-upper}]
For brevity, write $T = T(A)$.
We need to show that $\frac{\E[T\cdot\chi]}{\sqrt{\E[T^2]}} = \omega(\tau)$.
Without loss of generality, we may condition on $z_1 = 0$ throughout the proof, because the distribution of $A$ does not change if we condition on any realization of $z_1$.
Let $\E$ and $\p$ be the expectation and the probability with respect to the conditional distribution respectively.

Using $\chi = \bbone\{\dist(0,z_2) \le \tau/2\}$ and Proposition~\ref{prop:recovery-expectation}, we obtain
\begin{align*}
\E\kr{T \chi} = \int_{-\tau/2}^{\tau/2} \E\kr{T\cond z_2} \, \rmd z_2
\ge \tau \, \E\kr{ T \cond z_2 = \tau/2}
= \tau \binom{n-2}{\ell} \tau^{\ell} \lambda^{\ell+1} C_1 
\ge c_2 \, n^\ell \tau^{\ell+1} \lambda^{\ell+1} ,
\end{align*}
where $C_1 = C_1(\ell) = \int_{\frac{\ell}{2}}^{\frac{\ell}{2}+1}f_\ell(t)\rmd t$ with $\PDFIH_\ell$ defined in \eqref{eq:def-ih-pdf}, and $c_2 = c_2(\ell) > 0$.

Next, we have
$$
\E[T^2] = \E\kr{ \E[T^2 \mid z_2] }
= \E\kr{\Var(T\cond z_2)} + \E\kr{(\E[T\cond z_2])^2} .
$$
By Proposition~\ref{prop:recovery-expectation} again, 
\begin{align*}
\E\kr{(\E[T\cond z_2])^2}
&=\int_{-\frac{(\ell+1)\tau}{2}}^{\frac{(\ell+1)\tau}{2}} (\E[T \mid z_2])^2 \, \rmd z_2 \\
&\le (\ell+1) \tau \, (\E[T \mid z_2 = 0])^2 \\
&= (\ell+1) \tau \bigg[ \binom{n-2}{\ell} \tau^{\ell} \lambda^{\ell+1} C_3 \bigg]^2 \\
&\le C_4 \, n^{2\ell} \tau^{2\ell + 1} \lambda^{2\ell + 2} ,
\end{align*}
where $C_3 = C_3(\ell) = \int_{\frac{\ell}{2}-\frac{1}{2}}^{\frac{\ell}{2}+\frac{1}{2}}f_\ell(t)\rmd t$, and $C_4 = C_4(\ell) > 0$.
Moreover, recall that $p$ and $q$ are of the same order by assumption, so for any realization of $z_2$, Proposition~\ref{prop:recovery-variance} gives
$$
\Var(T \mid z_2)
\le C_5 \left( n^\ell + n^{2\ell - 1} \tau^{2 \ell - 2} \lambda^{2 \ell} + n^{\ell + 1/2} \tau \lambda^3 \right)
$$
for a constant $C_5 = C_5(\ell) > 0$.
Therefore, for a constant $C_6 = C_6(\ell) > 0$, 
\begin{align*}
\sqrt{\E[T^2]}
\le C_6 \left( n^\ell \tau^{\ell + 1/2} \lambda^{\ell+1} + n^{\ell/2} + n^{\ell - 1/2} \tau^{\ell - 1} \lambda^{\ell} + n^{\ell/2 + 1/4} \tau^{1/2} \lambda^{3/2} \right) .
\end{align*}

Combining the above bounds on $\E[T\chi]$ and $\sqrt{\E[T^2]}$, we conclude that
\begin{align*}
\frac{\E[T\cdot\chi]}{\sqrt{\E[T^2]}}
&\ge \frac{c_2}{C_6} \cdot \frac{ n^{\ell} \tau^{\ell+1} \lambda^{\ell+1} }{ n^\ell \tau^{\ell + 1/2} \lambda^{\ell+1} + n^{\ell/2} + n^{\ell - 1/2} \tau^{\ell - 1} \lambda^{\ell} + n^{\ell/2 + 1/4} \tau^{1/2} \lambda^{3/2} } \\
&\ge c_7 \, \min\left\{ \tau^{1/2} , n^{\ell/2} \tau^{\ell+1} \lambda^{\ell+1} , n^{1/2} \tau^{2} \lambda , n^{\ell/2-1/4} \tau^{\ell + 1/2} \lambda^{\ell - 1/2} \right\} .
\end{align*}
For this bound to be of order $\omega(\tau)$, it suffices to have
$$
\tau = o(1), \qquad
n \tau^{2} \lambda^{2 + 2/\ell} = \omega(1), \qquad 
n \tau^2 \lambda^2 = \omega(1).
$$
Since $C q \le p \le C' q$ for $C' > C > 1$, we have $\lambda = \frac{p-q}{\sqrt{q(1-q)}} = \Theta(p^{1/2})$.
Therefore, the above conditions all hold by the assumptions $\ell > 1/\delta$ and \eqref{eq:recover-upper-cond}.
\end{proof}

\begin{proof}[Proof of Theorem~\ref{thm:exact-recovery}]
Similar to the proof of Theorem~\ref{thm:recover-upper}, we write $T = T(A)$ and condition on $z_1 = 0$ throughout the proof.
We start by rewriting the expectation as the sum of type I and type II errors:
    \begin{equation*}
        \E\kr{(\hat \chi - \chi)^2} = \Pb\br{\hat\chi \neq \chi} = \Pb\br{\chi=1,\hat\chi=0}+\Pb\br{\chi=0,\hat\chi=1}.
    \end{equation*}
Since $\chi = \bbone\{-\tau/2 \le z_2 \le \tau/2\}$ and $\hat \chi = \bbone\{ T < \kappa \}$, we have
    \begin{align*}
        \Pb\br{\chi=1,\hat\chi=0} &= \int_{-\frac{\tau}{2}}^{\frac{\tau}{2}} \Pb\br{T<\kappa\cond z_2} \, \rmd z_2, \\
        \Pb\br{\chi=0,\hat\chi=1} 
        &= \int_{\frac{\tau}{2} \le |z_2| \le 1} \Pb\br{T\geq \kappa\cond z_2} \, \rmd z_2 \\
        &\le 2 \eps + \int_{\frac{\tau}{2} + \eps \le |z_2| \le \frac{(\ell+1) \tau}{2}} \Pb\br{T\geq \kappa\cond z_2} \, \rmd z_2 + \int_{\frac{(\ell+1) \tau}{2} < |z_2| \le 1} \Pb\br{T\geq \kappa\cond z_2} \, \rmd z_2.
    \end{align*}
It remains to bound the above three integrals:
\begin{itemize}
    \item 
Consider $z_2 \in [-\tau/2, \tau/2]$.
Let $\Delta(\eps)$ be defined in \eqref{eq:def-delta-epsilon} and $\kappa$ be defined in \eqref{eq:def-kappa}.
By Proposition~\ref{prop:recovery-expectation} and Chebyshev's inequality, 
$$
\Pb\br{T< \kappa \cond z_2}
\le \Pb\br{\abs{T-\E[T\cond z_2]}> \frac{\Delta(\epsilon)}{2} \; \Big| \; z_2}
\le \frac{4 \Var(T \mid z_2)}{\Delta(\eps)^2} .
$$

Lemma~\ref{lem:separation lemma} and Proposition~\ref{prop:recovery-variance} together imply that
\begin{align*}
\frac{\Var(T \mid z_2)}{\Delta(\eps)^2} 
&\le C_1 \frac{n^\ell + n^{2\ell - 1} \tau^{2 \ell - 2} \lambda^{2 \ell} + n^{\ell + \frac 12} \tau \lambda^3}{n^{2\ell} \eps^2 \tau^{2\ell-2} \lambda^{2\ell+2}} \\
&\le C_1 \left( \frac{1}{n^{\ell} \eps^2 \tau^{2\ell-2} \lambda^{2\ell+2}} + \frac{1}{n \eps^2  \lambda^{2}} + \frac{1}{n^{\ell-1/2} \eps^2 \tau^{2\ell-3} \lambda^{2\ell-1}} \right) 
\end{align*}
for a constant $C_1 = C_1(\ell) > 0$.
Using the assumptions $n \tau^2 \lambda^2 = \Theta(n \tau^2 p) \ge n^\delta$, $\ell > 3/\delta$, and $\eps = \tau n^{-\delta/4}$, we can check
\begin{equation}
\frac{\Var(T \mid z_2)}{\Delta(\eps)^2}
\le 3 C_1 \, n^{-\delta/2} .
\label{eq:che-var-bd}
\end{equation}

Therefore,
$$
\int_{-\frac{\tau}{2}}^{\frac{\tau}{2}} \Pb\br{T<\kappa\cond z_2} \, \rmd z_2
\le 12 C_1 \tau n^{-\delta/2} .
$$

\item
For $\frac{\tau}{2} + \eps \le |z_2| \le \frac{(\ell+1) \tau}{2}$, again, by Proposition~\ref{prop:recovery-expectation}, Chebyshev's inequality, and \eqref{eq:che-var-bd},
$$
\Pb\br{T \ge \kappa \cond z_2}
\le \Pb\br{\abs{T-\E[T\cond z_2]}> \frac{\Delta(\epsilon)}{2} \mid z_2}
\le \frac{4 \Var(T \mid z_2)}{\Delta(\eps)^2} 
\le 12 C_1 n^{-\delta/2} .
$$
Therefore,
$$
\int_{\frac{\tau}{2} + \eps \le |z_2| \le \frac{(\ell+1) \tau}{2}} \Pb\br{T\geq \kappa\cond z_2} \, \rmd z_2
\le 12 (\ell+1) C_1 \tau n^{-\delta/2} .
$$

\item
For $\frac{(\ell+1) \tau}{2} < |z_2| \le 1$, we have $\E[T \mid z_2] = 0$ by Proposition~\ref{prop:recovery-expectation}.
Combining Chebyshev's inequality, the second bound in Proposition~\ref{prop:recovery-variance}, \eqref{eq:def-kappa}, and Proposition~\ref{prop:recovery-expectation}, we obtain
\begin{align*}
\Pb\br{T\geq \kappa\cond z_2} 
\le \frac{\Var(T\cond z_2)}{\kappa^2}
&\le C_2 \frac{n^\ell + n^{2\ell - 1} \tau^{2 \ell - 1} \lambda^{2 \ell} + n^{\ell + \frac 12} \tau^2 \lambda^3}{n^{2\ell} \tau^{2\ell} \lambda^{2\ell+2}} \\
&\le C_2 \left( \frac{1}{n^{\ell} \tau^{2\ell} \lambda^{2\ell+2}} + \frac{1}{n \tau  \lambda^{2}} + \frac{1}{n^{\ell-1/2} \tau^{2\ell} \lambda^{2\ell-1}} \right) 
\end{align*}
for $C_2 = C_2(\ell) > 0$.
Using the assumptions $n \tau^2 \lambda^2 = \Theta(n \tau^2 p) \ge n^\delta$ and $\ell > 3/\delta$, we can check
$$
\Pb\br{T\geq \kappa\cond z_2}
\le 3 C_2 \, \tau n^{-\delta/2} .
$$
Therefore,
$$
\int_{\frac{(\ell+1) \tau}{2} < |z_2| \le 1} \Pb\br{T\geq \kappa\cond z_2} \, \rmd z_2 \le 3 C_2 \, \tau n^{-\delta/2} .
$$
\end{itemize}
In summary, we have obtained $\E\kr{(\hat \chi - \chi)^2} = \Pb\br{\hat\chi \neq \chi} \le C_3 \tau n^{-\delta/2}$ for $C_3 = C_3(\ell) > 0$.
\end{proof}

\section*{Acknowledgments}
C.M.\ was supported in part by NSF grants DMS-2053333 and DMS-2210734. Part of this work was done while A.S.W.\ was with the Algorithms and Randomness Center at Georgia Tech, supported by NSF grants CCF-2007443 and CCF-2106444. S.Z.\ was supported in part by NSF grant DMS-2053333. We thank Guy Bresler, Will Perkins, and Jiaming Xu for helpful discussions. We thank the anonymous referees for constructive comments.

\bibliographystyle{alpha}
\bibliography{main}

\end{document}